\documentclass[11pt]{amsart}

\usepackage{amsmath,amssymb,amsfonts,amsthm}
\usepackage{enumitem}
\usepackage{microtype}
\usepackage{graphicx}
\usepackage{subcaption}
\usepackage{booktabs}
\usepackage{float}
\usepackage{placeins}
\usepackage{xcolor}
\usepackage{hyperref}
\usepackage[numbers]{natbib}
\usepackage[nameinlink,capitalize]{cleveref}

\hypersetup{hidelinks}

\newtheorem{theorem}{Theorem}[section]
\newtheorem{proposition}[theorem]{Proposition}

\newtheorem{remark}[theorem]{Remark}

\newcommand{\BTFPINN}{\operatorname{BTF\mbox{-}PINN}}

\begin{document}


\title{BTF-PINN: Enforcing Dirichlet Boundary Conditions Without Boundary Training}

\author{Wenyu Dong}
\address{Institute of Computational Mathematics and Scientific/Engineering Computing, Academy of Mathematics and System Sciences, Chinese Academy of Sciences, Beijing 100190, China}
\email{dongwenyu@lsec.cc.ac.cn}

\author{Shuo Zhang$^{\ast}$}
\thanks{*Corresponding author.}
\address{State Key Laboratory of Mathematical Sciences (SKLMS) and State Key Laboratory of Scientific and Engineering Computing (LSEC), Institute of
	Computational Mathematics and Scientific/Engineering Computing, Academy of Mathematics and Systems Science, Chinese Academy of Sciences,
	100190, Beijing, China}
\address{School of Mathematical Sciences, University of Chinese Academy of Sciences, 100049, Beijing, China}
\email{szhang@lsec.cc.ac.cn}

\begin{abstract}
		The homogeneous Dirichlet boundary value problem captures the core difficulty of solving Dirichlet problems with non-interpolatory methods. We propose BTF-PINN (Boundary-Training-Free Physics-Informed Neural Network), an interior-only strategy for solving homogeneous Dirichlet boundary value problems that requires no boundary training, boundary penalties, or boundary-conforming parametrizations. The key idea is to embed the essential boundary condition into a newly designed boundary-free loss function. We prove the equivalence between the proposed interior-only variational formulation and the original boundary value problem, establish a sharp threshold condition for the residual weight, and develop a convergence analysis based on the coercivity of the functional. Numerical experiments on high-dimensional problems, irregular geometries, and anisotropic elliptic equations demonstrate the effectiveness of BTF-PINN. Comparisons with standard boundary-penalty PINNs further show that BTF-PINN achieves superior boundary trace accuracy.	
	\end{abstract}
	
	\maketitle
	
	\begin{center}
\end{center}

\section{Introduction}
	\label{sec:introduction}
	
	Neural-network solvers provide a mesh-free representation of PDE solutions and
	have been useful in settings where high dimensionality or complicated geometry
	makes classical discretization expensive. Representative approaches include
	physics-informed neural networks (PINNs) \cite{raissi2019physics}, the Deep Ritz
	method \cite{eyu2018deepritz}, Deep Galerkin methods
	\cite{sirignano2018dgm}, weak and variational PINNs
	\cite{kharazmi2021hpvpinn}, and stochastic formulations for high-dimensional
	problems \cite{han2018solving}. Despite their different objectives, these
	methods share a common feature: the trial function is represented globally by a
	neural network rather than by boundary-fitted local basis functions.
	
	For Dirichlet problems, this flexibility comes with a familiar difficulty.
	A generic neural network does not satisfy prescribed boundary values, and the
	essential boundary condition must thus be incorporated through an
	additional mechanism. Typical choices include boundary collocation and penalty
	terms, distance-function or hard-constraint constructions, Nitsche-type
	formulations, and boundary-trace-based variational approaches. Penalty
	formulations require a balance between interior and boundary losses and may
	exhibit gradient-scale mismatch \cite{wang2021understanding}. Hard
	constraints and distance-function ansatzes depend on geometric information
	\cite{sukumar2022exact,berrone2023enforcing}, while Nitsche-type neural methods
	retain boundary integrals and stabilization parameters
	\cite{liao2021deep}. Boundary norms of \(H^{1/2}\)-type provide formulations
	consistent with trace regularity and offer another perspective for incorporating
	Dirichlet information \cite{liu2023deepritz,kim2025trace}. All these methods rely on explicit boundary training, boundary penalties, or boundary-conforming parametrizations.
	
	In contrast, natural boundary conditions can be directly incorporated into variational neural solvers: a Neumann condition appears in the
	Euler--Lagrange equation of the energy and therefore requires no
	separate boundary penalty. Such natural subproblems can consequently be
	handled by the Deep Ritz method in an unconstrained trial space. Recently, a key connection between natural and essential boundary conditions was revealed in \cite{yu2025natural} through the de~Rham complex. The connection can be showed via this model problem. Consider the Laplace equation on a 2D simply connected domain $\Omega$:
	\[
	-\Delta \tilde{u}=0\quad\text{in }\Omega,\qquad
	u=g\quad\text{on }\Gamma:=\partial\Omega .
	\]
	Since ${\rm div}\nabla\tilde{u}=0$, we can rewrite $\nabla\tilde{u}=\operatorname{curl}\phi$ for some $\phi\in H^1(\Omega)$ determined by
	\[
	(\operatorname{curl}\phi,\operatorname{curl}\psi)_{L^2(\Omega)}
	=
	\langle \partial_t g,\psi\rangle_\Gamma,
	\qquad
	\forall\psi\in H^1(\Omega),
	\]
	where \(\partial_tg\) denotes the weak tangential derivative of the boundary data, and further
	\[
	(\nabla\tilde{u},\nabla v)_{L^2(\Omega)}
	=
	-(\operatorname{curl}\phi,\nabla v)_{L^2(\Omega)},
	\qquad
	\forall v\in H^1(\Omega).
	\]
	Consequently, solving for $\tilde{u}$ (equivalently $\nabla\tilde{u}$) is reduced to solving a harmonic equation with a natural boundary condition for $\phi$, from which $\operatorname{curl}\phi$ is then recovered. This connection leads to a natural Deep Ritz method for essential boundary value problems, first developed in 2D in \cite{yu2025natural}, subsequently extended to higher dimensions in \cite{chen2026natdrm}, and generalized to classical meshfree methods in \cite{zhang2026ndm}. Boundary penalties and boundary-conforming parametrizations are avoided.  
	
	A nonhomogeneous Dirichlet problem can be decomposed into a homogeneous problem with the same source term and a harmonic lifting. As the preceding discussion shows, the homogeneous Dirichlet boundary value problem captures the core difficulty of solving Dirichlet problems with non-interpolatory methods. Accordingly, this paper focuses on the homogeneous Dirichlet problem and develops a boundary-training-free PINN method for it.

	The main results of this work are summarized as follows.
	\begin{enumerate}[label=(\roman*),leftmargin=2.2em]
		\item We prove that the homogeneous Dirichlet Poisson problem admits an
		equivalent boundary-free variational formulation based on the newly-designed functional \eqref{eq:functional}.
		\item We extend the minimization principle to symmetric uniformly elliptic
		matrix-valued divergence-form operators, with the intrinsic threshold
		determined by the first weighted Dirichlet eigenvalue.
		\item We establish a convergence analysis framework for the proposed variational
		formulation based on the coercivity of the loss function and approximation properties of
		the trial class.
		\item We test the method on oscillatory, high-dimensional, irregular-domain,
		and anisotropic heat-conduction problems, and examine the threshold behavior
		through a controlled sweep of the residual weight.  A four-dimensional
		interface problem is included as a hybrid extension in which only the
		internal transmission conditions are imposed explicitly.
	\end{enumerate}

	The remainder of this paper is organized as follows.  Section~\ref{sec:theory} proves the
	boundary-free minimization result, its sharpness, and the extension to uniformly elliptic matrix-valued coefficients.  Section~\ref{sec:btf-pinn} develops the convergence
	analysis.  Section~\ref{sec:numerics} presents the numerical experiments.
	Section~\ref{sec:conclusion} concludes the paper.
	
	\FloatBarrier

\section{A boundary-free equivalent formulation of Dirichlet problems}
	\label{sec:theory}

	\subsection{A boundary-free optimization problem for the Poisson equation}
	
	Let \(\Omega\subset\mathbb R^d\) be a bounded connected Lipschitz domain.  We
	write
	\[
	\|v\|_0=\|v\|_{L^2(\Omega)},\qquad
	|v|_1=\|\nabla v\|_{L^2(\Omega)}.
	\]
	Consider the Dirichlet problem
	\begin{equation}
		\begin{cases}
			-\Delta u=f, & \text{in } \Omega,\\
			u=0, & \text{on } \partial\Omega .
		\end{cases}
		\label{eq:poisson}
	\end{equation}
	The boundary-unconstrained strong-residual space is
	\begin{equation}
		\mathcal V
		=
		\bigl\{v\in H^1(\Omega):\Delta v\in L^2(\Omega)\bigr\},
		\label{eq:trial-space}
	\end{equation}
	where the Laplacian is understood in the distributional sense.  We assume that \eqref{eq:poisson} admits a solution
	\[
	u\in \mathcal V\cap H^1_0(\Omega).
	\]
	For \(C>0\),
	define
	\begin{equation}
		\mathcal J_f^C(v)
		=
		C\|\Delta v+f\|_0+|v|_1,
		\qquad v\in\mathcal V .
		\label{eq:functional}
	\end{equation}
	The functional involves only interior quantities and is invariant under the
	addition of a constant.  Consequently, any equivalent formulation must account
	for this constant ambiguity.
	
	Let \(C_P\) denote the sharp Dirichlet Poincaré constant of \(\Omega\), defined by
	\begin{equation}
		C_P
		:=
		\sup_{0\neq w\in H^1_0(\Omega)}
		\frac{\|w\|_0}{|w|_1}
		=
		\frac{1}{\sqrt{\lambda_1^D(\Omega)}} .
		\label{eq:poincare-constant}
	\end{equation}
	Then
	\begin{equation}
		\|w\|_0\le C_P |w|_1,
		\qquad
		\forall w\in H^1_0(\Omega).
		\label{eq:poincare}
	\end{equation}
	
	\begin{theorem}[Boundary-free equivalent formulation]
		\label{thm:minimization-property}
		Let \(u\in \mathcal V\cap H^1_0(\Omega)\) solve \eqref{eq:poisson}. If
		\(C>C_P\), then
		\[
		\mathcal J_f^C(v)\geq \mathcal J_f^C(u),
		\qquad
		\forall v\in\mathcal V .
		\]
		Equality holds if and only if \(v-u\) is a constant.
	\end{theorem}
	
\begin{proof}
	Fix \(v\in\mathcal V\). By the harmonic decomposition of
	\(H^1(\Omega)\), we write
	\[
	v-u=z+h,
	\]
	where \(z\in H_0^1(\Omega)\) and \(h\in\mathcal H\), with
	\[
	\mathcal H
	=
	\left\{
	w\in H^1(\Omega):
	\int_\Omega\nabla w\cdot\nabla\phi\,dx=0,
	\quad
	\forall\phi\in H_0^1(\Omega)
	\right\}.
	\]
	The decomposition is orthogonal with respect to the \(H^1\)-seminorm, namely,
	\[
	(\nabla h,\nabla\phi)=0,\qquad \forall\phi\in H_0^1(\Omega).
	\]
	Since \(h\in\mathcal H\), it is weakly harmonic, and together with
	\(-\Delta u=f\), we have
	\(\Delta v+f=\Delta z\) in the distributional sense.
	
	For \(z\in H_0^1(\Omega)\) with \(\Delta z\in L^2(\Omega)\), integration by
	parts and the Poincar\'e inequality give
	\[
	|z|_1^2=-(\Delta z,z)
	\leq
	\|\Delta z\|_0\|z\|_0
	\leq
	C_P\|\Delta z\|_0|z|_1 ,
	\]
	and hence \(|z|_1\leq C_P\|\Delta z\|_0\).
	
	Since \(u+z\in H_0^1(\Omega)\), the weak harmonicity of \(h\) yields
	\[
	|v|_1^2=|u+z+h|_1^2=|u+z|_1^2+|h|_1^2 .
	\]
	Therefore,
	\[
	|v|_1\geq |u+z|_1\geq |u|_1-|z|_1 .
	\]
	Combining this estimate with the residual identity gives
	\begin{equation*}
		\mathcal J_f^C(v)
		=
		C\|\Delta z\|_0+|v|_1
		\geq
		C\|\Delta z\|_0+|u|_1-|z|_1
		\geq
		|u|_1+(C-C_P)\|\Delta z\|_0 .
	\end{equation*}
	Since \(\mathcal J_f^C(u)=|u|_1\), we obtain
	\(\mathcal J_f^C(v)\geq\mathcal J_f^C(u)\).
	
	Now suppose that equality holds. Since \(C>C_P\), the above estimate
	implies \(\Delta z=0\) in the distributional sense. Together with
	\(z\in H_0^1(\Omega)\), this gives \(z=0\). Hence \(v-u=h\).
	The orthogonality relation then gives
	\[
	|v|_1^2=|u|_1^2+|h|_1^2 .
	\]
	Because equality of the functionals implies
	\(|v|_1=|u|_1\), we have \(|h|_1=0\). Thus \(h\) is constant on the
	connected domain \(\Omega\), and consequently \(v-u=c\) for some
	\(c\in\mathbb R\).
	
	Conversely, if \(v-u\) is a constant, then
	\(\Delta v+f=0\) in the distributional sense and
	\(|v|_1=|u|_1\), which gives
	\(\mathcal J_f^C(v)=\mathcal J_f^C(u)\).
\end{proof}
	
	\begin{remark}
	\ref{thm:minimization-property} shows that, for \(C>C_P\), the homogeneous
	Dirichlet problem \eqref{eq:poisson} is characterized by the boundary-free minimization problem
	\[
	\inf_{v\in\mathcal V}
	\Bigl[
	C\|\Delta v+f\|_0
	+
	|v|_1
	\Bigr],
	\]
	unique up to an additive constant. Note that the infimum is achievable. 
\end{remark}

	\begin{proposition}[Sharpness of the residual threshold]
		\label{prop:sharp-threshold}
		The condition $C > C_P$ established in 
		\ref{thm:minimization-property} is sharp, ensuring uniform recovery of the solution regardless of the source term $f$.  Let
		\((\lambda_1^D,\phi_1)\) be a first Dirichlet eigenpair,
		\[
		-\Delta\phi_1=\lambda_1^D\phi_1,
		\qquad
		\phi_1\in H^1_0(\Omega),
		\qquad
		C_P=\frac{1}{\sqrt{\lambda_1^D}},
		\]
		and consider \eqref{eq:poisson} with
		\[
		u=\phi_1,
		\qquad
		f=\lambda_1^D\phi_1.
		\]
		Then:
		\begin{enumerate}[label=(\roman*),leftmargin=2.2em]
			\item if \(C=C_P\), every
			\[
			v=t\phi_1+c,
			\qquad 0\le t\le1,
			\qquad c\in\mathbb R,
			\]
			is a minimizer of \(\mathcal J_f^C\);
			\item if \(C<C_P\), the Dirichlet solution \(u=\phi_1\) is not a
			minimizer of \(\mathcal J_f^C\).
		\end{enumerate}
	\end{proposition}
	
	\begin{proof}
		For \(v=t\phi_1+c\),
		\[
		\Delta(t\phi_1)+f
		=
		(1-t)\lambda_1^D\phi_1,
		\qquad
		|t\phi_1|_1
		=
		|t|\sqrt{\lambda_1^D}\,
		\|\phi_1\|_{L^2(\Omega)}.
		\]
		Hence
		\[
		\mathcal J_f^C(t\phi_1+c)
		=
		\left(
		C\lambda_1^D|1-t|
		+
		\sqrt{\lambda_1^D}|t|
		\right)
		\|\phi_1\|_{L^2(\Omega)}.
		\]
		If \(C=C_P=1/\sqrt{\lambda_1^D}\), then for \(0\le t\le1\),
		\[
		\mathcal J_f^C(t\phi_1+c)
		=
		\sqrt{\lambda_1^D}\,
		\|\phi_1\|_{L^2(\Omega)}
		=
		\mathcal J_f^C(\phi_1).
		\]
		The lower bound in the proof of
		\ref{thm:minimization-property} remains valid at \(C=C_P\), so
		\(\mathcal J_f^C(v)\ge\mathcal J_f^C(\phi_1)\) for all
		\(v\in\mathcal V\).  Therefore all functions above are minimizers.
		
		If \(C<C_P\), then
		\[
		\mathcal J_f^C(0)
		=
		C\lambda_1^D
		\|\phi_1\|_{L^2(\Omega)}
		<
		\sqrt{\lambda_1^D}
		\|\phi_1\|_{L^2(\Omega)}
		=
		\mathcal J_f^C(\phi_1),
		\]
		which proves the second assertion.
	\end{proof}
	
	\begin{remark}[An explicit upper bound for the threshold $C_P$]
		For the optimal Dirichlet Poincar\'e constant,
		\begin{equation}
			C_P=\frac{1}{\sqrt{\lambda_1^D(\Omega)}},
		\end{equation}
		where \(\lambda_1^D(\Omega)\) is the first Dirichlet eigenvalue of
		\(-\Delta\). The Li--Yau lower bound (\cite{liyau1983schrodinger})
		\[
		\lambda_1^D(\Omega)
		\ge
		\frac{d}{d+2}C_d|\Omega|^{-2/d},
		\qquad
		C_d=(2\pi)^2B_d^{-2/d},
		\]
		where \(B_d=\pi^{d/2}/\Gamma(d/2+1)\) is the volume of the unit ball in
		\(\mathbb R^d\), yields the explicit estimate
		\begin{equation}\label{rem:li-yau-bound}
			C_P
			\le
			\left(\frac{d+2}{dC_d}\right)^{1/2}
			|\Omega|^{1/d}.
		\end{equation}
		Thus, the Li--Yau inequality (\cite{liyau1983schrodinger}) provides a computable sufficient condition on the residual weight for general domains.
	\end{remark}
	
	\begin{remark}[Practical choice of the residual weight]
		\label{rem:practical-C}
		The theory requires \(C>C_P\).  On a general domain where the Dirichlet
		Poincar\'e constant is not known analytically, several practical strategies
		are available:
		\begin{enumerate}[label=(\roman*),leftmargin=2.2em]
			\item On the unit cube \((0,1)^d\), the threshold is explicit:
			\(C_\ast=(\pi\sqrt d)^{-1}\).  A choice \(C\approx 1.5\,C_\ast\)
			provides a practical margin, as demonstrated by the numerical sweep in
			Section~\ref{sec:threshold}.
			\item For a general Lipschitz domain, the Li--Yau bound
			\eqref{rem:li-yau-bound} gives an upper bound \(C_P^{\rm ub}\) for
			\(C_P\).  Choosing \(C\) slightly above this conservative bound, e.g.\
			\(C = 1.2\,C_P^{\rm ub}\), satisfies the sufficient condition.
			\item Alternatively, \(C_P\) can be estimated numerically by solving a
			coarse finite-element eigenvalue problem for \(\lambda_1^D(\Omega)\) and
			using \eqref{eq:poincare-constant}.
		\end{enumerate}
	\end{remark}

	\subsection{Extension to uniformly elliptic matrix-valued coefficients}

	Let \(A\in L^\infty(\Omega;\mathbb R^{d\times d})\) be symmetric and
	uniformly elliptic, namely,
	\[
	\alpha|\xi|^2
	\leq
	\xi^{\top}A(x)\xi
	\leq
	\beta|\xi|^2
	\qquad
	\text{for a.e. }x\in\Omega,
	\quad
	\xi\in\mathbb R^d .
	\]
	We define the divergence-form operator
	\[
	\mathcal L_Av:=\nabla\cdot(A\nabla v),
	\]
	the weighted \(H^1\)-seminorm
	\[
	|v|_A
	:=
	\left(
	\int_\Omega A\nabla v\cdot\nabla v\,dx
	\right)^{1/2},
	\]
	and the corresponding strong-residual space
	\[
	\mathcal V_A
	:=
	\{v\in H^1(\Omega):\mathcal L_Av\in L^2(\Omega)\}.
	\]
	Let
	\[
	\lambda_1^A
	:=
	\inf_{0\neq w\in H_0^1(\Omega)}
	\frac{|w|_A^2}{\|w\|_0^2},
	\qquad
	C_A:=(\lambda_1^A)^{-1/2}.
	\]
	For the functional
	\[
	\mathcal J_{A,q}^C(v)
	=
	C\|\mathcal L_Av+q\|_0
	+
	|v|_A ,
	\]
	we have the following extension.
	
	\begin{proposition}[Extension to uniformly elliptic matrix-valued coefficients]
		\label{prop:variable-coefficient-extension}
		Let \(u\in\mathcal V_A\cap H_0^1(\Omega)\) solve
		\[
		-\mathcal L_Au=q
		\quad\text{in }\Omega .
		\]
		If \(C>C_A\), then
		\[
		\mathcal J_{A,q}^C(v)\geq
		\mathcal J_{A,q}^C(u),
		\qquad
		\forall v\in\mathcal V_A ,
		\]
		and equality holds if and only if \(v-u\) is a constant.
		Moreover,
		\[
		C_A\leq \frac{C_P}{\sqrt{\alpha}} .
		\]
	\end{proposition}
	
	\begin{proof}
		Fix \(v\in\mathcal V_A\). By the harmonic decomposition associated
		with the elliptic operator \(\mathcal L_A\), we write
		\[
		v-u=z+h_A,
		\]
		where \(z\in H_0^1(\Omega)\) and \(h_A\) belongs to the weakly
		\(A\)-harmonic space
		\[
		\mathcal H_A
		=
		\left\{
		w\in H^1(\Omega):
		\int_\Omega A\nabla w\cdot\nabla\phi\,dx=0,
		\quad
		\forall\phi\in H_0^1(\Omega)
		\right\}.
		\]
		Since \(h_A\) is weakly \(A\)-harmonic and
		\(-\mathcal L_Au=q\), we have
		\(\mathcal L_Av+q=\mathcal L_Az\) in the distributional sense.
		
		By the definition of \(C_A\),
		\(\|z\|_0\leq C_A|z|_A\). Therefore,
		\[
		|z|_A^2
		=-(\mathcal L_Az,z)
		\leq
		\|\mathcal L_Az\|_0\|z\|_0
		\leq
		C_A\|\mathcal L_Az\|_0|z|_A ,
		\]
		and hence
		\[
		|z|_A\leq C_A\|\mathcal L_Az\|_0 .
		\]
		
		The \(A\)-orthogonality of \(h_A\) gives
		\[
		|v|_A^2=|u+z|_A^2+|h_A|_A^2 .
		\]
		Hence
		\[
		|v|_A\geq |u+z|_A\geq |u|_A-|z|_A .
		\]
		Consequently,
		\[
		\begin{aligned}
			\mathcal J_{A,q}^C(v)
			&=
			C\|\mathcal L_Az\|_0+|v|_A\\
			&\geq
			|u|_A+(C-C_A)\|\mathcal L_Az\|_0 .
		\end{aligned}
		\]
		Since \(\mathcal J_{A,q}^C(u)=|u|_A\), the minimality follows.
		
		Assume now that equality holds. Since \(C>C_A\), we obtain
		\(\mathcal L_Az=0\). Together with \(z\in H_0^1(\Omega)\), this implies
		\[
		|z|_A^2=-(\mathcal L_Az,z)=0,
		\]
		and hence \(z=0\). Thus \(v-u=h_A\). The orthogonality relation yields
		\[
		|v|_A^2=|u|_A^2+|h_A|_A^2 .
		\]
		Because \(|v|_A=|u|_A\), we have \(|h_A|_A=0\). By uniform ellipticity,
		\(h_A\) is constant on the connected domain \(\Omega\), and therefore
		\(v-u\) is constant.
		
		Conversely, if \(v-u\) is constant, then
		\(\mathcal L_Av+q=0\) and \(|v|_A=|u|_A\), which gives
		\(\mathcal J_{A,q}^C(v)=\mathcal J_{A,q}^C(u)\).
		
		Finally, for \(w\in H_0^1(\Omega)\), uniform ellipticity and the Poincar\'e
		inequality imply
		\[
		|w|_A^2\geq
		\alpha\|\nabla w\|_0^2
		\geq
		\frac{\alpha}{C_P^2}\|w\|_0^2 .
		\]
		Hence \(\lambda_1^A\geq\alpha/C_P^2\), and therefore
		\(C_A\leq C_P/\sqrt{\alpha}\).
	\end{proof}
	
	\begin{remark}[Recovery of the additive constant in practice]
		\label{rem:constant-recovery}
		The additive constant ambiguity is intrinsic to any formulation that depends
		only on \(\Delta v\) and \(\nabla v\).  In applications where the absolute
		level of the solution carries physical meaning, the constant must be fixed by
		an additional scalar condition.  Several options are available:
		\begin{enumerate}[label=(\roman*),leftmargin=2.2em]
			\item Prescribe the mean value \(\frac{1}{|\Omega|}\int_\Omega v\,dx = m_0\),
			as is common when fixing the gauge in pure Neumann problems.
			\item Supply one pointwise reference value \(v(x_0)=u_0\) at a location
			where the solution is known, e.g.\ from a sensor or a symmetry condition.
			\item In time-dependent extensions, the initial condition naturally fixes
			the constant.
		\end{enumerate}
		In the numerical experiments of this paper, the constant is determined by
		post-training alignment with the exact solution for diagnostic purposes
		(Section~\ref{sec:numerics}).  This is a validation protocol, not a
		requirement of the method.
	\end{remark}

	\FloatBarrier

\section{A boundary-training-free physics-informed neural network framework}
	\label{sec:btf-pinn}
	In this section, we consider a neural-network trial class as a particular
	realization of the approximation space and introduce the
	\emph{boundary-training-free physics-informed neural network} (BTF-PINN).
	
	Let
	\[
	\mathcal N_\theta
	=
	\{v_\theta(\cdot;\vartheta):\vartheta\in\Theta\}
	\]
	be a fully connected neural-network trial class, where \(\vartheta\)
	denotes the trainable parameters. 
	
	The BTF-PINN problem is defined as
	\begin{equation}
		\text{find } v_\theta^\ast \in\mathcal N_\theta\text{ such that }
		\mathcal J_f^C(v_\theta^\ast)
		=
		\min_{v_\theta\in\mathcal N_\theta}
		C\|\Delta v_\theta+f\|_0
		+
		|v_\theta|_1.
		\label{eq:btf-pinn-discrete-problem}
	\end{equation}
	
	\begin{remark}
		The global minimizer in \eqref{eq:btf-pinn-discrete-problem} is not necessarily
		attained for general neural-network trial classes.
		For any
		\(\varepsilon>0\), we consider an \(\varepsilon\)-suboptimal global minimizer
		\(v_\theta^\varepsilon\in\mathcal N_\theta\) satisfying
		\[
		\mathcal J_f^C(v_\theta^\varepsilon)
		\leq
		\inf_{v_\theta\in\mathcal N_\theta}
		\mathcal J_f^C(v_\theta)+\varepsilon .
		\]
		Such an \(\varepsilon\)-optimal solution always exists by the definition of
		the infimum. This framework is standard in non-convex optimization, where global minimizers may fail to exist and $\varepsilon$-optimal solutions provide a natural alternative; see, e.g., \cite{houska2012global}.
	\end{remark}	
	
	The main result of this section is the C\'ea-type estimate below.
	\begin{theorem}[Convergence analysis for BTF-PINN]
		\label{thm:btf-pinn-cea}
		
		Assume \(C>C_P\). Let \(v_\theta^\varepsilon\) be an
		\(\varepsilon\)-suboptimal global minimizer of \(\mathcal J_f^C\) over
		\(\mathcal N_\theta\), and let \(u\) be the solution of
		\eqref{eq:poisson}. 
		There exists a constant \(K>0\), such that,
		\begin{equation}
			\inf_{c\in\mathbb R}
			\|v_\theta^\varepsilon-u-c\|_{H^1(\Omega)}
			\leq
			K
			\inf_{w_\theta\in\mathcal N_\theta}
			\left\{
			C\|\Delta(w_\theta-u)\|_0
			+
			|w_\theta-u|_1
			+
			\varepsilon
			+
			\left(
			C\|\Delta(w_\theta-u)\|_0
			+
			|w_\theta-u|_1
			+
			\varepsilon
			\right)^{1/2}
			\right\}.
			\label{eq:btf-pinn-cea}
		\end{equation}
		
	\end{theorem}
	We postpone the proof of Theorem \ref{thm:btf-pinn-cea} until after some technical preparation.

	\begin{remark}[Convexity of the variational problem]
		\label{rem:convexity}
		The minimization problem in
		\ref{thm:minimization-property} is convex at the function-space level.
		Indeed, for any \(v,w\in\mathcal V\) and \(t\in[0,1]\), the triangle
		inequality gives
		\begin{align*}
			\mathcal J_f^C(tv+(1-t)w)
			&=
			C\left\|
			t(\Delta v+f)+(1-t)(\Delta w+f)
			\right\|_0
			+
			\left\|
			t\nabla v+(1-t)\nabla w
			\right\|_0
			\\
			&\le
			Ct\|\Delta v+f\|_0+t\|\nabla v\|_0
			+C(1-t)\|\Delta w+f\|_0+(1-t)\|\nabla w\|_0
			\\
			&=
			t\mathcal J_f^C(v)+(1-t)\mathcal J_f^C(w).
		\end{align*}
		Thus, \(\mathcal J_f^C\) is convex, although generally nonsmooth, on
		\(\mathcal V\).  This statement concerns convexity with respect to the
		function \(v\).  
	\end{remark}


\section{Convergence analysis}

	\begin{proposition}
		\label{prop:ve-upper-bound}
		
		Let \(u\) denote the minimizer of the functional \(\mathcal J_f^C\) over
		\(\mathcal V\).
		Then, for any
		\(v,w\in\mathcal V\) satisfying
		\[
		\mathcal J_f^C(v)\le \mathcal J_f^C(w),
		\]
		we have
		\begin{equation}
			\mathcal J_f^C(v)-\mathcal J_f^C(u)
			\le
			C\|\Delta(w-u)\|_0
			+
			|w-u|_1 .
			\label{eq:ve-upper-bound}
		\end{equation}
		
	\end{proposition}

	\begin{proof}
		By the assumption \(\mathcal J_f^C(v)\leq\mathcal J_f^C(w)\), we have
		\[
		\mathcal J_f^C(v)-\mathcal J_f^C(u)
		\leq
		\mathcal J_f^C(w)-\mathcal J_f^C(u).
		\]
		Set \(e=w-u\). Since \(-\Delta u=f\), we have
		\(\Delta w+f=\Delta e\). Therefore,
		\[
		\mathcal J_f^C(w)-\mathcal J_f^C(u)
		=
		C\|\Delta e\|_0+|u+e|_1-|u|_1 .
		\]
		By the triangle inequality,
		\(|u+e|_1-|u|_1\leq |e|_1\), and hence
		\[
		\mathcal J_f^C(w)-\mathcal J_f^C(u)
		\leq
		C\|\Delta e\|_0+|e|_1 .
		\]
		Combining the above estimates yields
		\[
		\mathcal J_f^C(v)-\mathcal J_f^C(u)
		\leq
		C\|\Delta(w-u)\|_0+|w-u|_1 .
		\]
	\end{proof}

	\begin{proposition}
		\label{prop:variational-error}
		
		Assume \(C>C_P\). There exists a constant
		\(K>0\), such that, for any \(v\in\mathcal V\),
		\begin{equation}
			\inf_{c\in\mathbb R}
			\|v-u-c\|_{H^1(\Omega)}
			\leq
			K
			\left(
			\mathcal J_f^C(v)-\mathcal J_f^C(u)
			+
			(\mathcal J_f^C(v)-\mathcal J_f^C(u))^{1/2}
			\right).
			\label{eq:error-control}
		\end{equation}
		
	\end{proposition}
	
\begin{proof}
	
	Let
	\[
	v-u=z+h,
	\]
	where \(z\in H_0^1(\Omega)\) and \(h\in\mathcal H\) is the weakly harmonic
	component of \(v-u\). Since \(h\in\mathcal H\), we have
	\(\Delta v+f=\Delta z\) in the distributional sense. Hence,
	\[
	\mathcal J_f^C(v)-\mathcal J_f^C(u)
	=
	C\|\Delta z\|_0+|v|_1-|u|_1 .
	\]
	
	The orthogonality of the harmonic component gives
	\[
	|v|_1^2=|u+z|_1^2+|h|_1^2,
	\]
	and therefore
	\[
	|v|_1\geq |u+z|_1\geq |u|_1-|z|_1 .
	\]
	Consequently,
	\[
	\mathcal J_f^C(v)-\mathcal J_f^C(u)
	\geq
	C\|\Delta z\|_0-|z|_1
	\geq
	(C-C_P)\|\Delta z\|_0 .
	\]
	It follows that
	\[
	\|\Delta z\|_0
	\leq
	\frac{\mathcal J_f^C(v)-\mathcal J_f^C(u)}{C-C_P},
	\]
	and
	\[
	|z|_1
	\leq
	\frac{C_P}{C-C_P}
	(\mathcal J_f^C(v)-\mathcal J_f^C(u)).
	\]
	
	It remains to estimate the harmonic component. Since
	\(|v|_1\leq |u|_1+\mathcal J_f^C(v)-\mathcal J_f^C(u)\), we obtain
	\[
	\begin{aligned}
		|h|_1^2
		&=
		|v|_1^2-|u+z|_1^2\\
		&\leq
		\left(|u|_1+\mathcal J_f^C(v)-\mathcal J_f^C(u)\right)^2
		-\left(|u|_1-|z|_1\right)^2\\
		&\leq
		2|u|_1
		\left(
		\mathcal J_f^C(v)-\mathcal J_f^C(u)+|z|_1
		\right)
		+
		(\mathcal J_f^C(v)-\mathcal J_f^C(u))^2 .
	\end{aligned}
	\]
	Using the estimate for \(|z|_1\), we have
	\[
	|h|_1^2
	\leq
	K_1
	\left(
	\mathcal J_f^C(v)-\mathcal J_f^C(u)
	+
	(\mathcal J_f^C(v)-\mathcal J_f^C(u))^2
	\right).
	\]
	Hence,
	\[
	|h|_1
	\leq
	K_2
	\left(
	(\mathcal J_f^C(v)-\mathcal J_f^C(u))^{1/2}
	+
	\mathcal J_f^C(v)-\mathcal J_f^C(u)
	\right).
	\]
	
	Combining the estimates for \(z\) and \(h\) gives
	\[
	|z|_1+|h|_1
	\leq
	K
	\left(
	\mathcal J_f^C(v)-\mathcal J_f^C(u)
	+
	(\mathcal J_f^C(v)-\mathcal J_f^C(u))^{1/2}
	\right).
	\]
	Since \(v-u=z+h\), the Poincar\'e--Wirtinger inequality yields
	\[
	\inf_{c\in\mathbb R}\|v-u-c\|_{H^1(\Omega)}
	\leq
	K
	\left(
	\mathcal J_f^C(v)-\mathcal J_f^C(u)
	+
	(\mathcal J_f^C(v)-\mathcal J_f^C(u))^{1/2}
	\right).
	\]
	
\end{proof}

	\begin{proposition}
		\label{prop:cea-error-estimate}
		Assume \(C>C_P\). Let \(u\) denote the minimizer of the functional
		\(\mathcal J_f^C\) over \(\mathcal V\).
		There exists a constant \(K>0\), such that, for any \(v,w\in\mathcal V\)
		satisfying
		\[
		\mathcal J_f^C(v)\le \mathcal J_f^C(w),
		\]
		\begin{equation}
			\inf_{c\in\mathbb R}
			\|v-u-c\|_{H^1(\Omega)}
			\le
			K
			\left(
			C\|\Delta(w-u)\|_0
			+
			|w-u|_1
			+
			\left(
			C\|\Delta(w-u)\|_0
			+
			|w-u|_1
			\right)^{1/2}
			\right).
			\label{eq:cea-error-estimate}
		\end{equation}
		
	\end{proposition}
	
	\begin{proof}
		By Proposition ~\ref{prop:variational-error}, we have
		\[
		\inf_{c\in\mathbb R}
		\|v-u-c\|_{H^1(\Omega)}
		\leq
		K
		\left(
		\mathcal J_f^C(v)-\mathcal J_f^C(u)
		+
		(\mathcal J_f^C(v)-\mathcal J_f^C(u))^{1/2}
		\right).
		\]
		
		Since \(\mathcal J_f^C(v)\leq\mathcal J_f^C(w)\), 
		Proposition \ref{prop:ve-upper-bound} yields
		\[
		\mathcal J_f^C(v)-\mathcal J_f^C(u)
		\leq
		C\|\Delta(w-u)\|_0+|w-u|_1 .
		\]
		
		Substituting this estimate into the above inequality gives
		\[
		\inf_{c\in\mathbb R}
		\|v-u-c\|_{H^1(\Omega)}
		\leq
		K
		\left(
		C\|\Delta(w-u)\|_0
		+
		|w-u|_1
		+
		\left(
		C\|\Delta(w-u)\|_0
		+
		|w-u|_1
		\right)^{1/2}
		\right).
		\]
	\end{proof}
	
	\begin{remark}
		Propositions \ref{prop:ve-upper-bound}, \ref{prop:variational-error} and \ref{prop:cea-error-estimate} establish, respectively, the boundedness, coercivity, and quasi-monotonicity of $\mathcal J_f^C$. In particular, Proposition \ref{prop:cea-error-estimate} shows that if \(\mathcal J_f^C(v)-\mathcal J_f^C(u)\le \mathcal J_f^C(w)-\mathcal J_f^C(u)\), then the $H^1$-error of $v-u$ (modulo constants) is controlled by a combination of $\|\Delta(w-u)\|_0$ and $|w-u|_1$; we refer to this property as quasi-monotonicity.
	\end{remark}
	
	\paragraph{Proof of Theorem \ref{thm:btf-pinn-cea}}

	Using \ref{prop:ve-upper-bound} with
	\(v=v_\theta^\varepsilon\) and \(w=w_\theta\), together with the definition of
	\(v_\theta^\varepsilon\), and combining this estimate with \ref{prop:variational-error}, we obtain, for any
	\(w_\theta\in\mathcal N_\theta\),
	\[
	\begin{aligned}
		\inf_{c\in\mathbb R}
		\|v_\theta^\varepsilon-u-c\|_{H^1(\Omega)}
		\leq
		K\Bigg\{
		C\|\Delta(w_\theta-u)\|_0
		+
		|w_\theta-u|_1
		+
		\varepsilon
		+
		\left(
		C\|\Delta(w_\theta-u)\|_0
		+
		|w_\theta-u|_1
		+
		\varepsilon
		\right)^{1/2}
		\Bigg\}.
	\end{aligned}
	\]
	Taking the infimum over \(w_\theta\in\mathcal N_\theta\) yields
	\eqref{eq:btf-pinn-cea}.
	\FloatBarrier

\section{Numerical experiments}
	\label{sec:numerics}
	
	In this section, we use fully connected neural networks with the sine activation function to illustrate the validity of BTF-PINN. The experiments are divided into two parts.  We first demonstrate the fundamental properties of BTF-PINN, including its approximation accuracy, the threshold behavior of the residual weight, and the comparison with conventional boundary-penalty PINNs.  We then investigate the applicability of BTF-PINN to more general settings, including irregular domains, high-dimensional problems, matrix-coefficient elliptic equations, and interface problems.
	
	\subsection{Implementation and training protocol}
	\label{sec:experimental-setup}
	
	In all experiments, the BTF-PINN objective defined in
	Section~\ref{sec:btf-pinn} is evaluated by empirical quadrature.  Given
	sets of interior quadrature points \(\{x_i,w_i\}_{i=1}^{N_Q}\subset\Omega\), we define
	\[
	\|r\|_{L^2(\Omega),Q}
	=
	\left(
	\sum_{i=1}^{N_Q}w_i|r(x_i)|^2
	\right)^{1/2}.
	\]
	The corresponding practical training objective is
	\begin{equation}
		\mathcal J_{f,Q}^C(v_\theta)
		=
		C\|\Delta v_\theta+f\|_{L^2(\Omega),Q}
		+
		\|\nabla v_\theta\|_{L^2(\Omega),Q}.
		\label{eq:numerical-loss}
	\end{equation}
	The network remains unconstrained on \(\partial\Omega\), and \(C\) is the only
	additional weight introduced in the objective.
	
	In the implementation, the two unsquared norms in
	\eqref{eq:numerical-loss} are evaluated as stabilized square roots of
	quadrature averages:
	\[
	\|r\|_{L^2(\Omega),Q}
	=
	\left(
	\sum_{i=1}^{N_Q}w_i|r(x_i)|^2+\varepsilon_{\rm n}
	\right)^{1/2}.
	\]
	The parameter	$\varepsilon_{\rm n}=10^{-12}$ is introduced solely to avoid numerical singularities when the residual becomes
	very small.  It is several orders of magnitude below the reported losses and
	therefore does not affect the numerical results.
	
	The computations are performed in single precision using PyTorch
	with CUDA acceleration on an NVIDIA GeForce RTX 5060 GPU.
	
	Boundary points do not enter the BTF-PINN training objective.  They are used
	only for posterior diagnostics and, in Section~\ref{sec:trace-compare}, for
	the boundary-penalty PINN baseline.  For the hybrid interface problem,
	points on the internal interface are used to impose the transmission
	conditions, while the outer Dirichlet boundary remains absent from the
	training loss.
	
	The main network architectures and training configurations are summarized
	in Table~\ref{tab:training-configurations}. Here \(D\times W\) denotes a
	fully connected neural network with \(D\) hidden layers and width \(W\).
	Depending on the problem, optimization is performed either by an Adam-type
	method or by an Adam stage followed by L-BFGS refinement. Here, AdamW denotes
	the Adam optimizer with decoupled weight decay.
	
	\begin{table}[!htbp]
		\caption{Network architectures and principal traini configurations used in the numerical experiments. }
		\label{tab:training-configurations}
		\resizebox{\textwidth}{!}{%
			\begin{tabular}{lcccc}
				\toprule
				experiment
				& network
				& activation
				& interior sampling
				& optimizer \\
				\midrule
				
				3-D multi-frequency Poisson
				& \(4\times96\)
				& sine
				& Gauss--Legendre, \(64^3\) points
				& Adam + L-BFGS \\
				
				10-D Poisson
				& \(4\times128\)
				& sine
				& factorized Gauss, up to \(262144\) points
				& Adam + L-BFGS \\
				
				2-D irregular-domain Poisson
				& \(4\times96\)
				& sine
				& random interior batches of size \(32768\)
				& Adam \\
				
				3-D irregular-domain Poisson
				& \(4\times128\)
				& sine
				& random interior batches of size \(8192\)
				& Adam \\
				
				2-D matrix-coefficient heat equation
				& \(4\times96\)
				& sine
				& random interior batches of size \(8192\)
				& AdamW \\
				
				3-D matrix-coefficient heat equation
				& \(4\times96\)
				& sine
				& random interior batches of size \(4096\)
				& AdamW \\
				
				4-D interface problem
				& two \(4\times96\) subnetworks
				& sine
				& interior and interface samples
				& AdamW \\
				
				residual-weight sweep
				& \(4\times128\)
				& sine
				& random batches of size \(8192\)
				& Adam \\
				\bottomrule
			\end{tabular}%
		}
	\end{table}

	\subsection{Basic numerical examples of BTF-PINN}
	\paragraph{Multi-frequency solution in three dimensions}
	\label{sec:3d-multi}
	
	We first consider a three-dimensional nonseparable manufactured solution with
	frequencies up to five:
	\begin{align}
		u(x_1,x_2,x_3)
		={}&
		\sin(\pi x_1)\sin(\pi x_2)\sin(\pi x_3)                              
		+0.55\sin(2\pi x_1)\sin(3\pi x_2)\sin(\pi x_3)                       \\
		&-0.40\sin(3\pi x_1)\sin(\pi x_2)\sin(4\pi x_3)                    
		+0.30\sin(4\pi x_1)\sin(2\pi x_2)\sin(3\pi x_3)                     \\
		&-0.20\sin(5\pi x_1)\sin(4\pi x_2)\sin(2\pi x_3).
		\label{eq:3d-multifrequency-solution}
	\end{align}
	The right-hand side is computed analytically as \(f=-\Delta u\).  For \(d=3\),
	\[
	C_\ast=(\pi\sqrt3)^{-1}=0.183776,
	\]
	and we use
	\[
	C=0.24=1.306\,C_\ast .
	\]
	
	Since the functional is invariant under additive constants, the solution
	error is evaluated after removing the constant mode. We report the shifted
	relative \(L^2\) error
	\begin{equation}
		E_{L^2}^{\rm sh}
		=
		\inf_{c\in\mathbb R}
		\frac{\|v_\theta-u-c\|_{L^2(\Omega)}}
		{\|u\|_{L^2(\Omega)}} ,
	\end{equation}
	and the relative \(H^1\)-seminorm error
	\begin{equation}
		E_{H^1}
		=
		\frac{\|\nabla(v_\theta-u)\|_{L^2(\Omega)}}
		{|u|_1}.
	\end{equation}

	\begin{table}[!htbp]
		\centering
		\caption{Three-dimensional multi-frequency manufactured solution.}
		\label{tab:3d-multifrequency}
		\begin{tabular}{l c c c}
			\toprule
			method & \(C\) & \(E_{H^1}\) & \(E_{L^2}^{\rm sh}\) \\
			\midrule
			\(\BTFPINN\)
			& \(0.24\)
			& \(2.532\times10^{-2}\)
			& \(1.117\times10^{-2}\) \\
			\bottomrule
		\end{tabular}
	\end{table}
	
	Table~\ref{tab:3d-multifrequency} shows that the proposed interior loss
	approximates the multi-frequency field with a relative \(H^1\)-seminorm error
	of \(2.53\%\) and a shifted relative \(L^2\) error of \(1.12\%\).
	
	Figure~\ref{fig:3d-loss-curves} shows the training history.  The shifted
	interior errors decrease during Adam training and are further reduced by
	L-BFGS; the posterior trace diagnostic follows the same trend.
	
	\begin{figure}[!htbp]
		\centering
		\IfFileExists{3D_loss.png}{%
			\includegraphics[height=0.50\textheight,keepaspectratio]{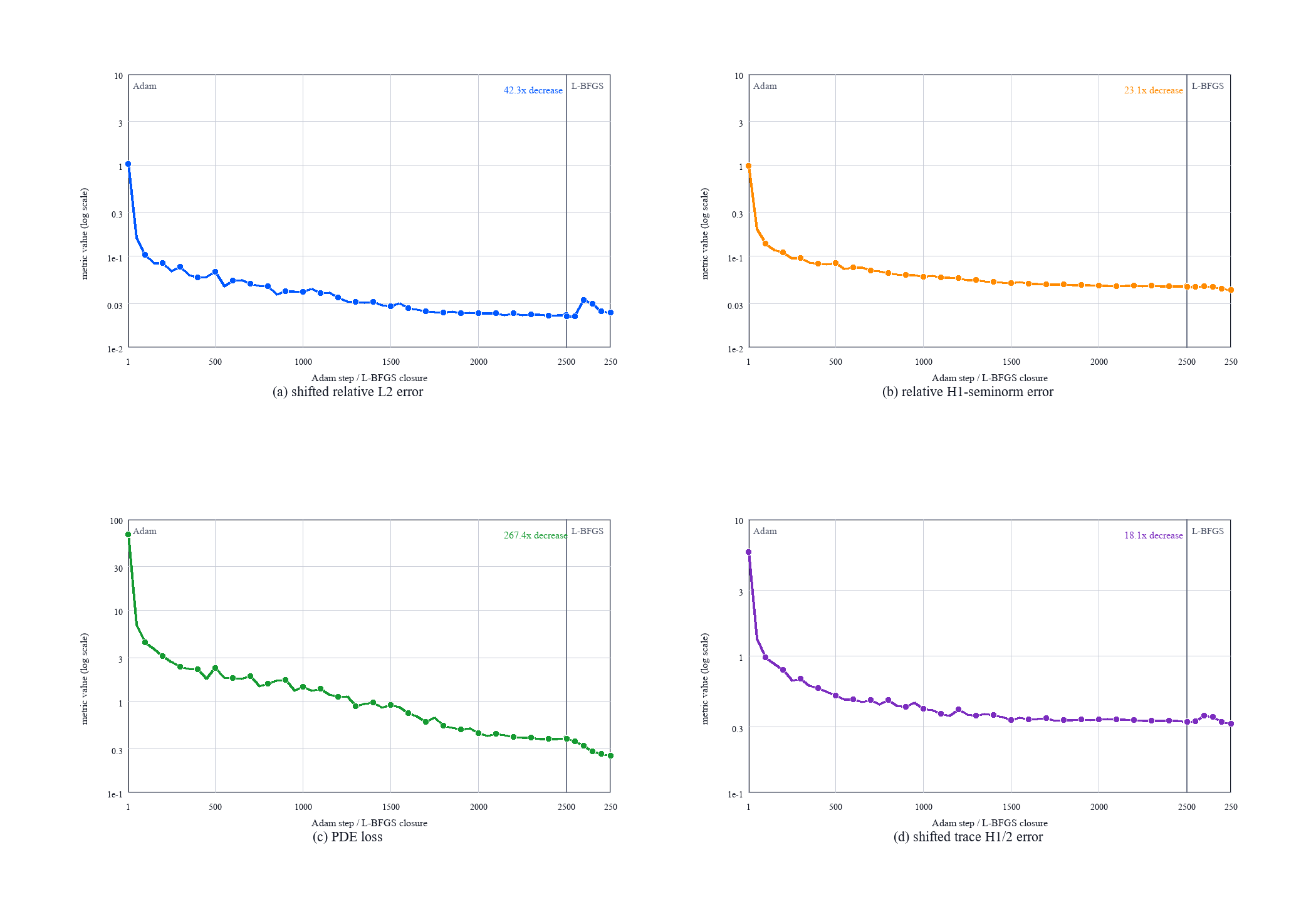}%
		}{%
			\fbox{\parbox{0.70\textwidth}{\centering Missing figure file:
					\texttt{3D\_loss.png}.}}%
		}
		\caption{Training history for the three-dimensional multi-frequency example.
			The panels show the shifted relative \(L^2\) error, the relative
			\(H^1\)-seminorm error, the PDE loss, and the shifted trace
			\(H^{1/2}\)-type error.}
		\label{fig:3d-loss-curves}
	\end{figure}
	
	\begin{figure}[!htbp]
		\centering
		\IfFileExists{fig_complex3d_multifreq_2607_high_contrast_stacks.pdf}{%
			\includegraphics[height=0.4\textheight,keepaspectratio]{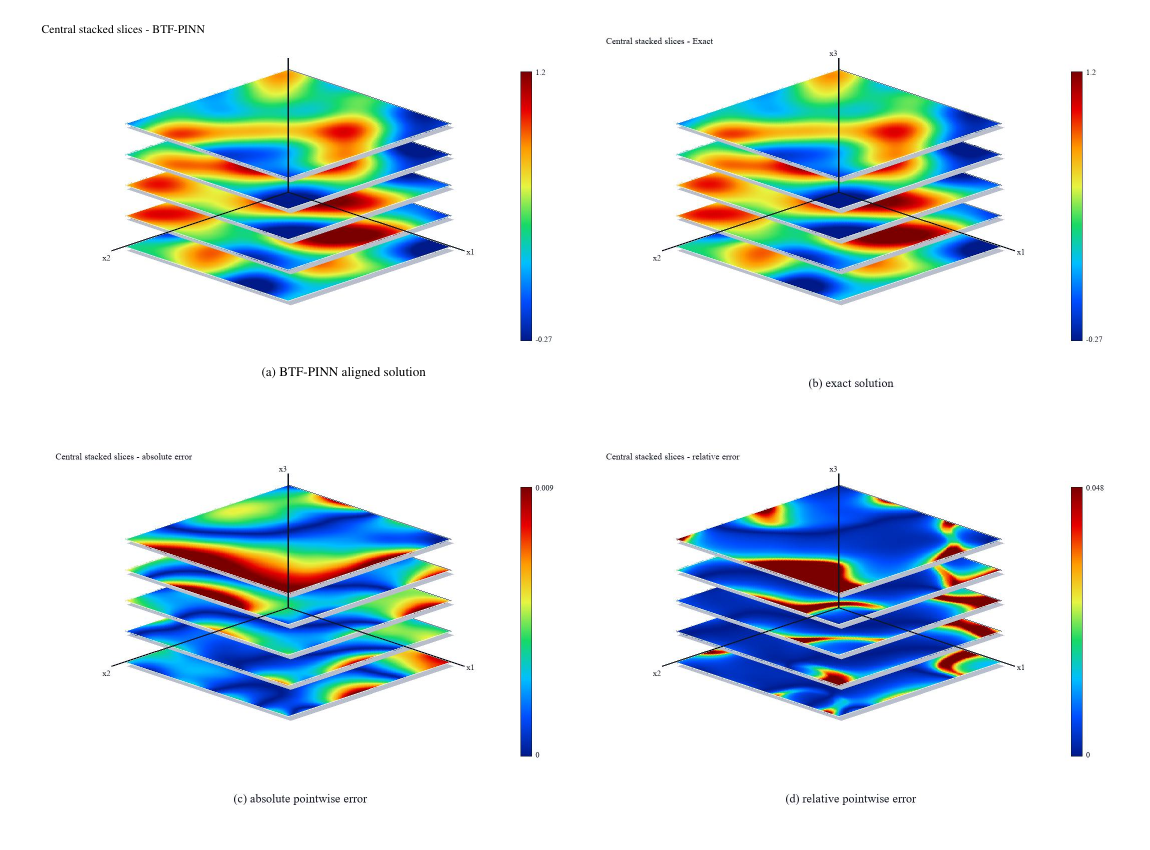}%
		}{%
			\fbox{\parbox{0.70\textwidth}{\centering Missing figure file:
					\texttt{fig\_complex3d\_multifreq\_2607\_high\_contrast\_stacks.pdf}.}}%
		}
		\caption{Stacked-slice diagnostics for the three-dimensional multi-frequency
			example.  The panels compare the aligned BTF-PINN prediction, the exact
			solution, the absolute pointwise error, and the pointwise relative error.}
		\label{fig:complex3d-stacked-slices}
	\end{figure}
	
	\begin{figure}[!htbp]
		\centering
		\IfFileExists{fig_3d_multifrequency_heatmap_grid.pdf}{%
			\includegraphics[height=0.4\textheight,keepaspectratio]{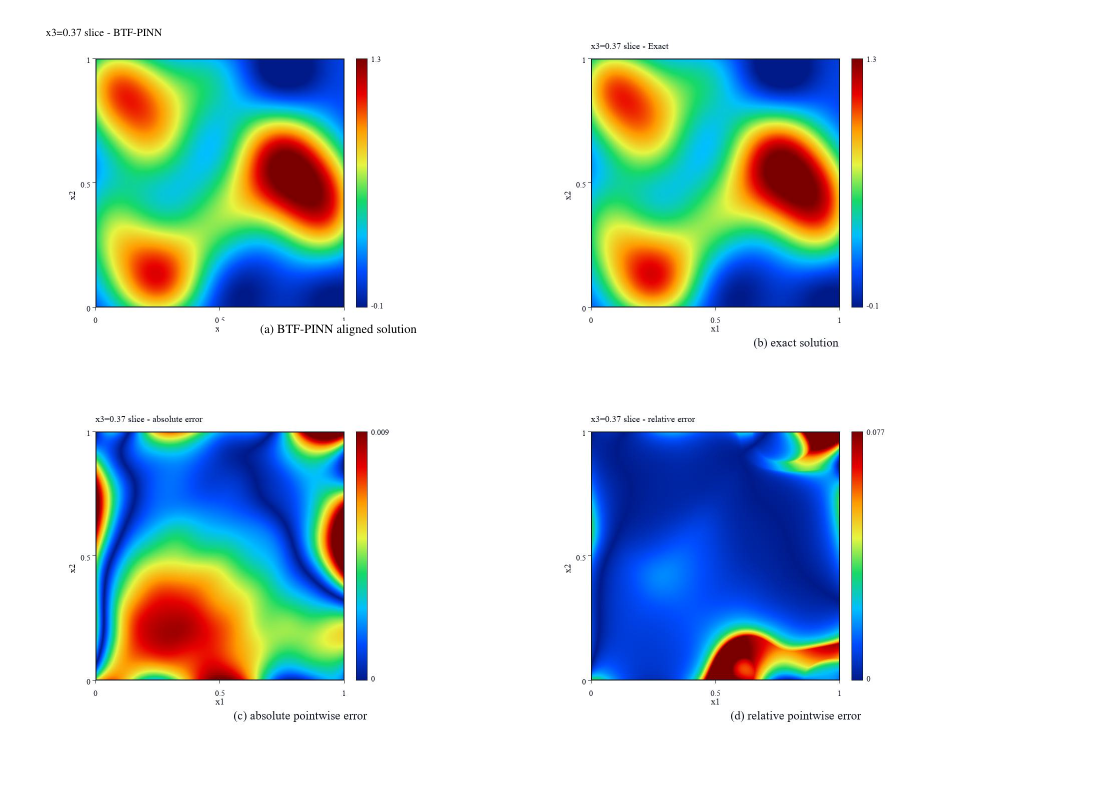}%
		}{%
			\fbox{\parbox{0.70\textwidth}{\centering Missing figure file:
					\texttt{fig\_3d\_multifrequency\_heatmap\_grid.pdf}.}}%
		}
		\caption{Heatmap diagnostics for the three-dimensional multi-frequency
			example.  The panels compare the aligned BTF-PINN prediction, the exact
			solution, the absolute pointwise error, and the pointwise relative error.}
		\label{fig:3d-heatmap-diagnostics}
	\end{figure}
	
	Figures~\ref{fig:complex3d-stacked-slices} and
	\ref{fig:3d-heatmap-diagnostics} provide complementary slice diagnostics for
	the nonseparable oscillatory field.  The errors remain localized across the
	displayed sections.
	
	\paragraph{Residual-weight threshold behavior}
	\label{sec:threshold}
	
	We next examine the effect of the residual weight \(C\).  The continuous
	theory identifies the Dirichlet Poincar\'e constant as the relevant threshold, and
	the strict condition \(C>C_P\) yields the equivalence in
	\ref{thm:minimization-property}.  On the unit cube this threshold is
	\(C_\ast=(\pi\sqrt d)^{-1}\).  We therefore test several values of \(C\) relative
	to \(C_\ast\) on the normalized product mode
	\[
	u(x)=2^{d/2}\prod_{j=1}^d\sin(\pi x_j),
	\qquad
	f(x)=d\pi^2u(x),
	\]
	in dimensions \(d=2,4,6,10\).  For every tested dimension we use the same protocol and take
	\[
	C/C_\ast\in\{0.8,\ 1.0,\ 1.2,\ 1.5,\ 2.0\},
	\qquad
	C_\ast=(\pi\sqrt d)^{-1},
	\]
	with three independent random seeds at each setting.
	
	A run is regarded as successful when the shifted relative $L^2$ error
	$E_{L^2}^{\rm sh}$ satisfies $E_{L^2}^{\rm sh}<0.2$.
	This criterion is used only for the success count; the errors themselves are
	reported in Table~\ref{tab:unified-threshold-sweep}.
	
	\begin{table}[!htbp]
		\centering
		\caption{Effect of the residual weight across dimensions.  A run is counted as
			successful if \(E_{L^2}^{\rm sh}<0.2\).  Entries are mean \(\pm\)
			standard deviation over three independent seeds.}
		\label{tab:unified-threshold-sweep}
		\scriptsize
		\setlength{\tabcolsep}{4pt}
		\renewcommand{\arraystretch}{0.92}
		\begin{tabular}{c c c c}
			\toprule
			\(d\) & \(C/C_\ast\) & \(E_{L^2}^{\rm sh}\) & success \\
			\midrule
			\(2\)  & \(0.8\) & \(5.938\times10^{-1}\pm1.452\times10^{-2}\) & \(0/3\) \\
			\(2\)  & \(1.0\) & \(6.067\times10^{-1}\pm8.980\times10^{-4}\) & \(0/3\) \\
			\(2\)  & \(1.2\) & \(2.078\times10^{-1}\pm3.531\times10^{-1}\) & \(2/3\) \\
			\(2\)  & \(1.5\) & \(5.232\times10^{-3}\pm8.655\times10^{-4}\) & \(3/3\) \\
			\(2\)  & \(2.0\) & \(8.423\times10^{-3}\pm1.477\times10^{-3}\) & \(3/3\) \\
			\midrule
			\(4\)  & \(0.8\) & \(7.535\times10^{-1}\pm7.542\times10^{-5}\) & \(0/3\) \\
			\(4\)  & \(1.0\) & \(7.059\times10^{-1}\pm6.617\times10^{-3}\) & \(0/3\) \\
			\(4\)  & \(1.2\) & \(3.394\times10^{-2}\pm1.189\times10^{-3}\) & \(3/3\) \\
			\(4\)  & \(1.5\) & \(3.799\times10^{-2}\pm1.458\times10^{-3}\) & \(3/3\) \\
			\(4\)  & \(2.0\) & \(4.208\times10^{-2}\pm4.854\times10^{-4}\) & \(3/3\) \\
			\midrule
			\(6\)  & \(0.8\) & \(8.459\times10^{-1}\pm1.027\times10^{-4}\) & \(0/3\) \\
			\(6\)  & \(1.0\) & \(8.433\times10^{-1}\pm4.134\times10^{-4}\) & \(0/3\) \\
			\(6\)  & \(1.2\) & \(6.414\times10^{-2}\pm1.184\times10^{-3}\) & \(3/3\) \\
			\(6\)  & \(1.5\) & \(6.962\times10^{-2}\pm1.083\times10^{-3}\) & \(3/3\) \\
			\(6\)  & \(2.0\) & \(7.429\times10^{-2}\pm6.237\times10^{-4}\) & \(3/3\) \\
			\midrule
			\(10\) & \(0.8\) & \(9.370\times10^{-1}\pm2.830\times10^{-4}\) & \(0/3\) \\
			\(10\) & \(1.0\) & \(9.369\times10^{-1}\pm2.484\times10^{-4}\) & \(0/3\) \\
			\(10\) & \(1.2\) & \(9.366\times10^{-1}\pm1.622\times10^{-4}\) & \(0/3\) \\
			\(10\) & \(1.5\) & \(1.592\times10^{-1}\pm1.454\times10^{-2}\) & \(3/3\) \\
			\(10\) & \(2.0\) & \(1.469\times10^{-1}\pm3.155\times10^{-3}\) & \(3/3\) \\
			\bottomrule
		\end{tabular}
	\end{table}
	
	Table~\ref{tab:unified-threshold-sweep} shows a clear dependence on the
	residual weight.  For \(C<C_\ast\), none of the tested runs reaches the
	prescribed accuracy in any dimension, and the shifted relative \(L^2\)-error
	remains large.  The borderline value \(C=C_\ast\) also does not achieve the
	target accuracy throughout the sweep.  This agrees with the fact that the
	equivalence theorem requires a strict margin above the threshold, while the
	trained neural model is affected by finite approximation, quadrature error,
	random initialization, and nonconvex optimization.
	
	When \(C\) is taken above \(C_\ast\), the shifted errors decrease substantially.
	In dimensions four and six, all three seeds already succeed at \(1.2C_\ast\).
	In dimension two, \(1.2C_\ast\) gives a mixed outcome: two seeds reach the
	prescribed accuracy range, while one remains outside, leading to the larger
	standard deviation reported in the table.  At \(1.5C_\ast\) and \(2C_\ast\),
	all dimension-two runs achieve the target accuracy with small shifted errors.
	The ten-dimensional case requires a larger practical margin:
	\(0.8C_\ast\), \(C_\ast\), and \(1.2C_\ast\) do not reach the target accuracy,
	whereas \(1.5C_\ast\) and \(2C_\ast\) achieve successful convergence for all
	three seeds.  This behavior can be explained by the additional numerical errors arising from neural approximation, empirical quadrature, and nonconvex optimization.
	As the dimension increases, these effects may become more pronounced, and a
	larger practical margin above \(C_\ast\) may therefore be required to obtain
	reliable convergence.
	
	The sweep examines whether the scale predicted by the continuous minimization
	result is reflected in the trained neural model.  The results indicate that
	\(C_\ast\) serves as a useful reference scale: below this value, the method does
	not reach the target accuracy in the tested cases, while above it the shifted
	error enters the desired range once a practical margin is introduced.  The
	required margin depends on dimension, approximation quality, quadrature,
	optimization, and the relative weighting of the residual and energy terms.
	
	\paragraph{Comparison with boundary-penalty PINNs}
	\label{sec:trace-compare}
	
	Although boundary information is not used in the BTF-PINN optimization, the
	resulting neural solution induces a posterior boundary trace.  We compare this
	boundary behavior with that of a standard boundary-penalty PINN.
	
	For the baseline method, we employ the objective
	\begin{equation}
		\mathcal L_{f}^{\beta}(\theta)
		=
		\|\Delta v_\theta+f\|_{L^2(\Omega)}^2
		+
		\beta
		\|\gamma_0v_\theta\|_{L^2(\partial\Omega)}^2 .
		\label{eq:boundary-penalty-pinn-loss}
	\end{equation}
	where the Dirichlet condition is imposed through sampled boundary values.
	The parameter \(\beta\) is selected from the discrete set
	\[
	\{0.01,\,0.1,\,1,\,10, \,100\}.
	\]
	The boundary-penalty PINN results reported in Table~\ref{tab:trace-detailed}
	correspond to the best-performing value among these candidates.

	For homogeneous Dirichlet problems, the exact boundary trace satisfies
	\(\gamma_0u=0\).  After training, we evaluate the recovered boundary trace
	\[
	g_\theta=\gamma_0v_\theta .
	\]
	Since the interior functional is invariant under additive constants, the
	boundary diagnostics are evaluated after removing the constant mode.  We define
	the shifted boundary \(L^2\) and \(H^{1/2}\) trace errors as
	\begin{equation}
		B_{L^2}^{\rm sh}
		=
		\inf_{c\in\mathbb R}
		\|g_\theta-c\|_{L^2(\partial\Omega)}
		\label{eq:boundary-shifted-l2}
	\end{equation}
	and
	\begin{equation}
		B_{1/2}^{\rm sh}
		=
		\inf_{c\in\mathbb R}
		\|g_\theta-c\|_{H^{1/2}(\partial\Omega)} .
		\label{eq:boundary-hhalf-shifted}
	\end{equation}
	These boundary quantities are evaluated only after training and do not enter
	the BTF-PINN optimization.
	
	The \(H^{1/2}(\partial\Omega)\) norm in
	\eqref{eq:boundary-hhalf-shifted} is computed through the Slobodeckij
	representation \cite{mclean2000strongly},
	\[
	\|g\|_{H^{1/2}(\partial\Omega)}^2
	=
	\|g\|_{L^2(\partial\Omega)}^2
	+
	\int_{\partial\Omega}\int_{\partial\Omega}
	\frac{|g(x)-g(y)|^2}{|x-y|^{d}}
	\,ds_x\,ds_y ,
	\]
	using boundary quadrature points independent of the interior training points.
	The same boundary quadrature set is used for both BTF-PINN and the
	boundary-penalty PINN baseline.

	\begin{table}[!htbp]
		\centering
		\caption{Shifted boundary trace diagnostics.  The shifted quantities remove
			the additive constant that cannot be identified by the interior objective.}
		\label{tab:trace-detailed}
		\begin{tabular}{c c c c}
			\toprule
			\(d\) & method & \(B_{L^2}^{\rm sh}\) & \(B_{1/2}^{\rm sh}\) \\
			\midrule
			\(2\) & \(\BTFPINN\) & \(9.628\times10^{-4}\) & \(1.165\times10^{-2}\) \\
			\(2\) & boundary-penalty PINN & \(1.231\times10^{-3}\) & \(1.479\times10^{-2}\) \\
			\midrule
			\(3\) & \(\BTFPINN\) & \(2.217\times10^{-3}\) & \(4.635\times10^{-2}\) \\
			\(3\) & boundary-penalty PINN & \(1.063\times10^{-2}\) & \(2.201\times10^{-1}\) \\
			\midrule
			\(4\) & \(\BTFPINN\) & \(4.670\times10^{-3}\) & \(1.023\times10^{-1}\) \\
			\(4\) & boundary-penalty PINN & \(1.716\times10^{-2}\) & \(3.462\times10^{-1}\) \\
			\midrule
			\(5\) & \(\BTFPINN\) & \(6.152\times10^{-3}\) & \(1.744\times10^{-1}\) \\
			\(5\) & boundary-penalty PINN & \(1.941\times10^{-2}\) & \(4.453\times10^{-1}\) \\
			\bottomrule
		\end{tabular}
	\end{table}

	\begin{figure}[!htbp]
		\centering
		\IfFileExists{fig_trace_hhalf_comparison.pdf}{%
			\includegraphics[width=0.58\textwidth]{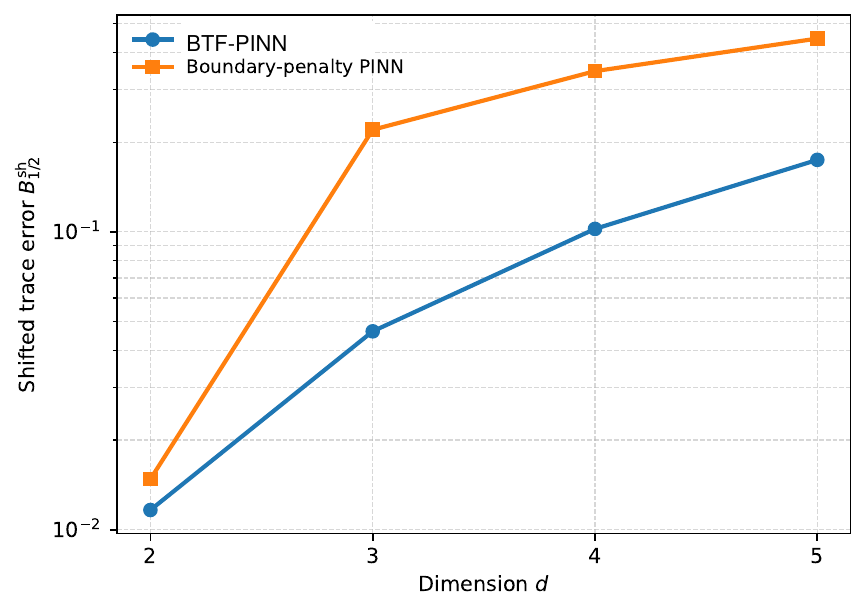}%
		}{%
			\fbox{\parbox{0.58\textwidth}{\centering Missing figure file:
					\texttt{fig\_trace\_hhalf\_comparison.pdf}.}}%
		}
		\caption{Dimension-wise comparison of shifted \(H^{1/2}\)-type boundary
			trace errors. Both methods are evaluated after removing the additive
			constant.}
		\label{fig:trace-hhalf-comparison}
	\end{figure}

	Table~\ref{tab:trace-detailed} and
	Figure~\ref{fig:trace-hhalf-comparison} show smaller shifted
	\(H^{1/2}\)-type trace errors for BTF-PINN in all tested dimensions.  In
	\(d=3\), for example, the error decreases from
	\(2.201\times10^{-1}\) to \(4.635\times10^{-2}\).  This behavior is consistent
	with the role of the interior Dirichlet energy in suppressing the nonconstant
	harmonic component of the trace.
	
	\subsection{Extended numerical examples of BTF-PINN}
	\label{sec:extended-numerical-examples}
	
	\paragraph{Irregular domains}
	\label{sec:irregular-domains}
	
	We next test the same boundary-training-free formulation on two non-tensor-product
	domains.  The domains are defined implicitly by smooth functions.  More
	precisely, we take
	\[
	\Omega_m=\{x\in[-1,1]^m:\phi_m(x)>0\},\qquad m=2,3,
	\]
	and the boundary is given by the zero level set \(\phi_m=0\).  This implicit
	description allows the same interior loss to be evaluated on irregular domains
	without introducing a boundary-conforming parametrization in the training
	objective.  A schematic illustration of the two irregular test domains is shown
	in Figure~\ref{fig:irregular-domain-schematic}.
	
	\begin{figure}[!htbp]
		\centering
		\IfFileExists{irrdomain.jpg}{%
			\includegraphics[width=0.6\textwidth]{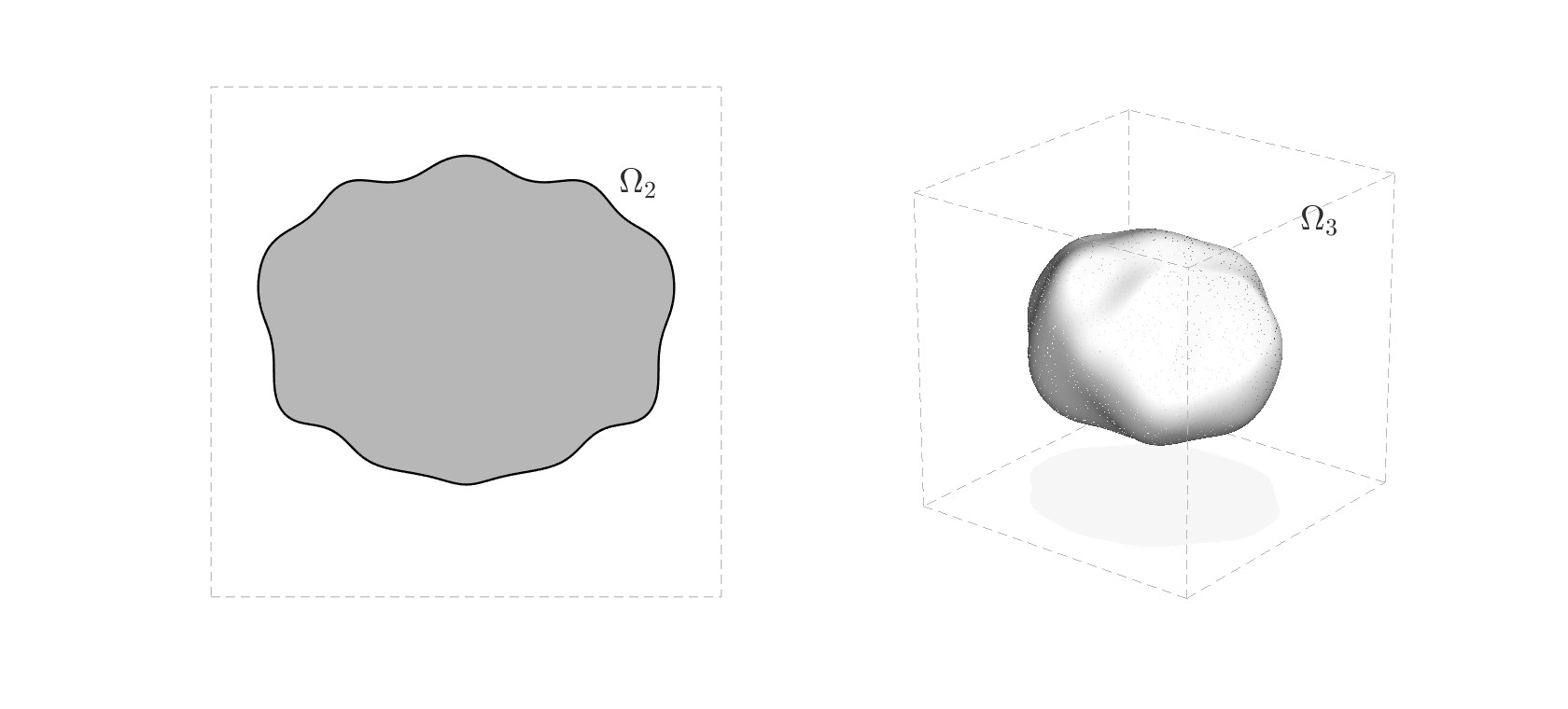}%
		}{%
			\fbox{\parbox{0.6\textwidth}{\centering Missing figure file:
					\texttt{irrdomain.jpg}.}}%
		}
		\caption{Schematic illustration of the irregular test domains.  The left panel illustrates the two-dimensional irregular domain, and the right
			panel illustrates the three-dimensional irregular domain.  The dashed boxes
			indicate the surrounding sampling boxes used to define the implicit domains.}
		\label{fig:irregular-domain-schematic}
	\end{figure}
	
	The manufactured solutions have the form \(u_m=\phi_m g_m\), and therefore
	vanish on \(\partial\Omega_m\).
	
	The two-dimensional domain is
	\[
	\Omega_2
	=
	\{(x,y)\in[-1,1]^2:\phi_2(x,y)>0\},
	\]
	where
	\[
	\phi_2(x,y)
	=
	1-\left(\frac{x}{0.82}\right)^2
	-\left(\frac{y-0.08}{0.62}\right)^2
	+0.12\cos(4\pi x)\cos(3\pi y).
	\]
	The exact solution is
	\[
	u_2(x,y)
	=
	\phi_2(x,y)
	\left[
	0.5
	+0.70\sin(2\pi x)\cos(\pi y)
	+0.35\cos(3\pi x-2\pi y)
	\right].
	\]
	
	The three-dimensional domain is
	\[
	\Omega_3
	=
	\{(x,y,z)\in[-1,1]^3:\phi_3(x,y,z)>0\},
	\]
	where
	\[
	\begin{aligned}
		\phi_3(x,y,z)
		={}&
		1-\left(\frac{x}{0.82}\right)^2
		-\left(\frac{y+0.04}{0.70}\right)^2
		-\left(\frac{z-0.06}{0.62}\right)^2  \\
		&+0.10\cos(3\pi x)\cos(2\pi y)
		+0.08\sin(2\pi z)\cos(2\pi x).
	\end{aligned}
	\]
	The exact solution is
	\[
	u_3(x,y,z)=\phi_3(x,y,z)g_3(x,y,z),
	\]
	with
	\[
	g_3(x,y,z)
	=
	0.45
	+0.45\sin(\pi x)\cos(2\pi y)
	+0.30\cos(2\pi x-\pi z)
	+0.22\sin(\pi y+2\pi z).
	\]
	In both cases the right-hand side is computed from the manufactured solution as
	\[
	f_m=-\Delta u_m .
	\]
	
	The training objective has the same form as in the previous experiments:
	\[
	\mathcal J_f^C(v_\theta)
	=
	C\|\Delta v_\theta+f\|_{L^2(\Omega),Q}
	+
	\|\nabla v_\theta\|_{L^2(\Omega),Q}.
	\]
	For these experiments, the residual weight is chosen from the
	Li--Yau estimate of the Dirichlet Poincar\'e constant.  Since the exact value of
	\(C_P\) is not available analytically on these implicit domains, we use the
	volume-based upper bound \(C_P^{\rm LY}\) given by
	\eqref{rem:li-yau-bound}, and take a \(20\%\) safety margin:
	\[
	C = 1.2\, C_P^{\rm LY}.
	\]
	This gives \(C=0.605\) for the two-dimensional irregular domain and
	\(C=0.456\) for the three-dimensional irregular domain.  This choice gives a
	conservative above-threshold residual weight without solving an auxiliary
	Dirichlet eigenvalue problem on the irregular domains.
	
	\begin{table}[!htbp]
		\centering
		\caption{Boundary-Training-Free runs on irregular domains. }
		\label{tab:irregular-domains}
		\begin{tabular}{c c c}
			\toprule
			case & \(E_{L^2}^{\rm sh}\) & \(E_{H^1}\)  \\
			\midrule
			2-D irregular
			& \(3.764\times10^{-3}\)
			& \(9.699\times10^{-3}\) \\
			3-D irregular
			& \(8.765\times10^{-3}\)
			& \(2.284\times10^{-2}\) \\
			\bottomrule
		\end{tabular}
	\end{table}
	
	The shifted relative \(L^2\)-error is
	\(3.764\times10^{-3}\) for the two-dimensional irregular domain and
	\(8.765\times10^{-3}\) for the three-dimensional one.  The corresponding
	relative \(H^1\)-seminorm errors are \(9.699\times10^{-3}\) and
	\(2.284\times10^{-2}\), respectively.

	\begin{figure}[!htbp]
		\centering
		\IfFileExists{fig_irregular2d_bfpinn_high_contrast.pdf}{%
			\includegraphics[width=0.45\textwidth,keepaspectratio]{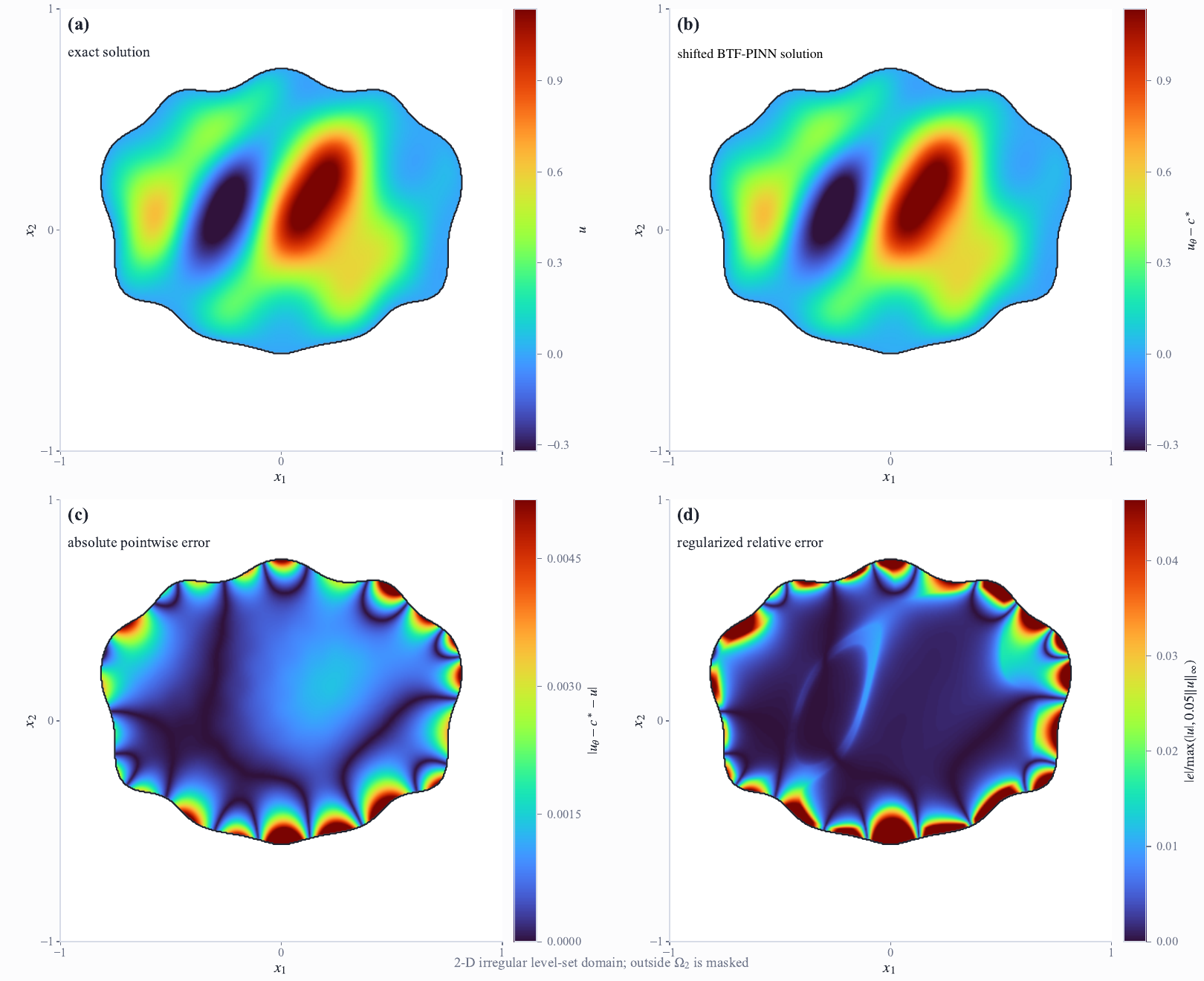}%
		}{%
			\fbox{\parbox{0.70\textwidth}{\centering Missing figure file:
					\texttt{fig\_irregular2d\_bfpinn\_high\_contrast.pdf}.}}%
		}
		\caption{Irregular two-dimensional domain.  The panels compare the exact
			solution, the shifted BTF-PINN prediction, the absolute pointwise error, and
			the regularized relative pointwise error.  The exterior of the domain is
			masked, and the black curve marks the boundary \(\phi_2=0\).}
		\label{fig:irregular2d}
	\end{figure}
	
	\begin{figure}[!htbp]
		\centering
		\IfFileExists{fig_irregular3d_slice_z0_high_contrast.pdf}{%
			\includegraphics[width=0.45\textwidth,keepaspectratio]{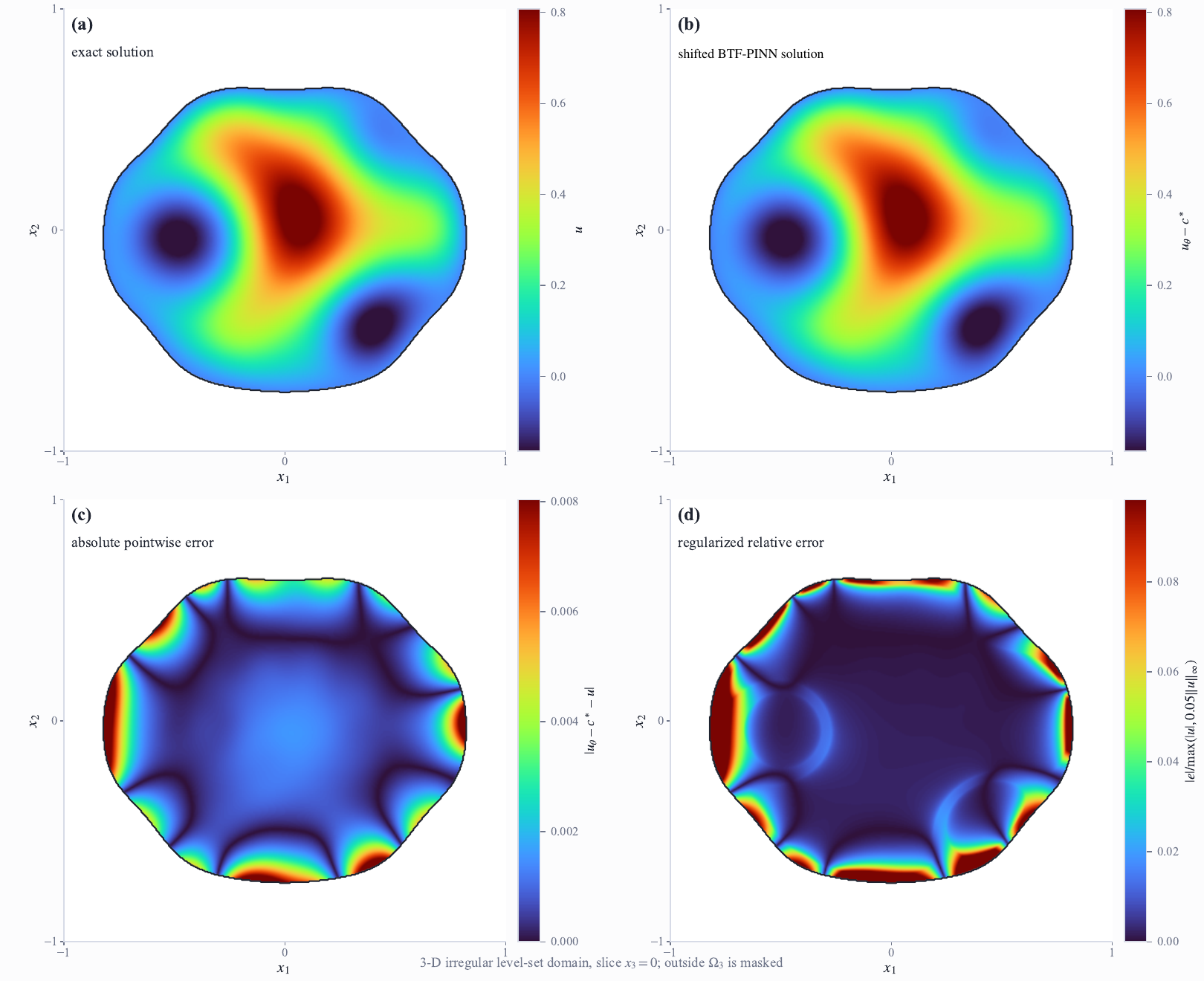}%
		}{%
			\fbox{\parbox{0.60\textwidth}{\centering Missing figure file:
					\texttt{fig\_irregular3d\_slice\_z0\_high\_contrast.pdf}.}}%
		}
		\caption{Central slice of the irregular three-dimensional domain.  The panels
			compare the exact solution, the shifted BTF-PINN prediction, the absolute
			pointwise error, and the regularized relative pointwise error on the slice
			\(z=0\).  The exterior of the domain is masked, and the black curve marks
			the slice of the boundary \(\phi_3=0\).}
		\label{fig:irregular3d-slice}
	\end{figure}
	
	\begin{figure}[!htbp]
		\centering
		\IfFileExists{fig_irregular3d_multislice_stacked.pdf}{%
			\includegraphics[height=0.34\textheight,keepaspectratio]{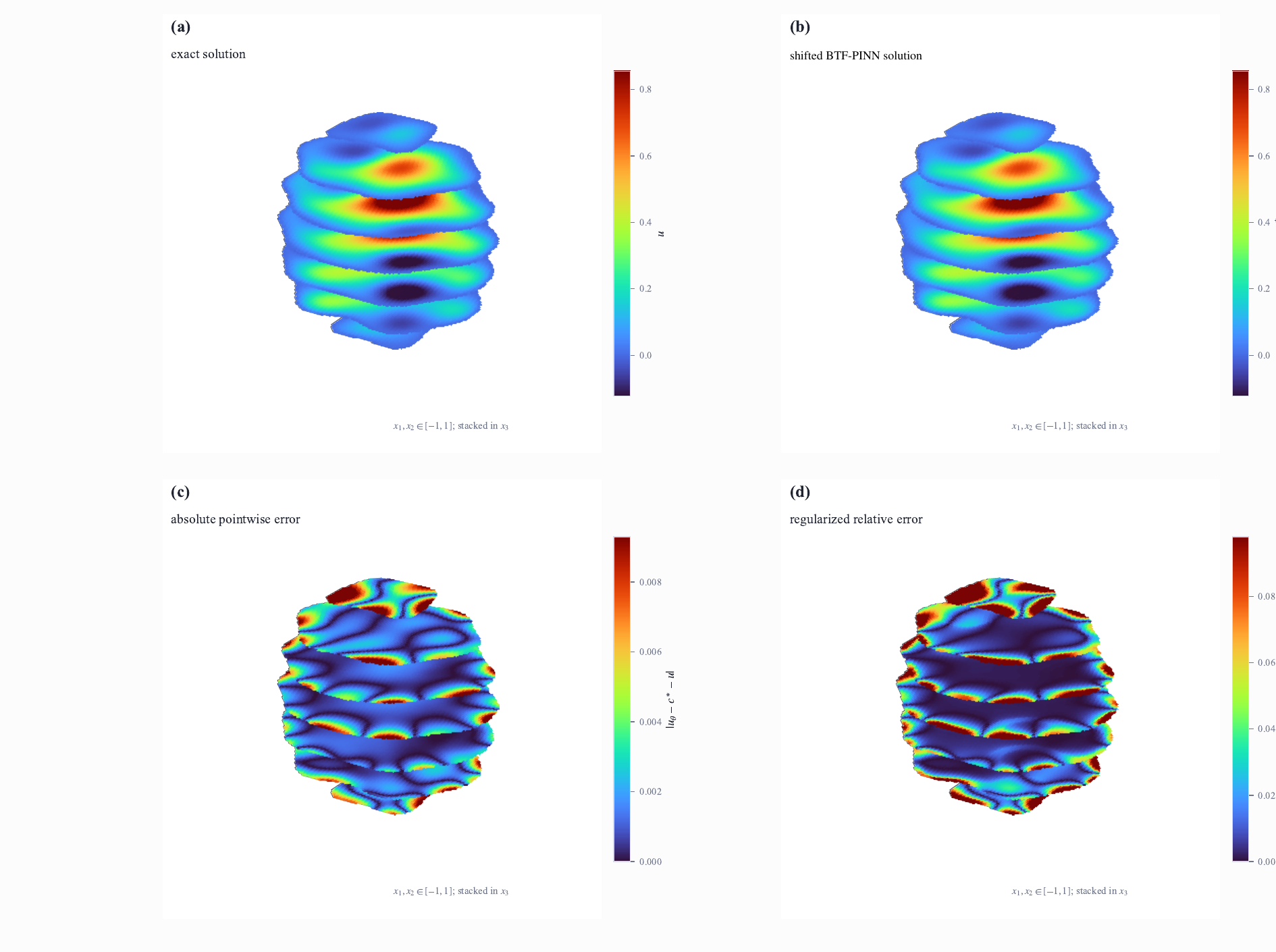}%
		}{%
			\fbox{\parbox{0.70\textwidth}{\centering Missing figure file:
					\texttt{fig\_irregular3d\_multislice\_stacked.pdf}.}}%
		}
		\caption{Stacked-slice visualization of the irregular three-dimensional
			domain.  Several \((x,y)\)-slices are shown along the \(z\)-direction.  Each
			slice is masked outside \(\Omega_3\), and the black curves mark the
			corresponding traces of \(\phi_3=0\).}
		\label{fig:irregular3d-stacked}
	\end{figure}
	
	\paragraph{Matrix-coefficient heat conduction on irregular domains}
	\label{sec:numerics-variable-heat}
	
	We next test the matrix-coefficient formulation of
	\ref{prop:variable-coefficient-extension} on the irregular domains introduced
	in \ref{sec:irregular-domains}.  The empirical loss is obtained from
	\(\mathcal J_{A,q}^C\) by replacing the continuous norms with interior
	quadrature.
	
	For the two-dimensional problem,
	\[
	A_2(x,y)
	=
	\begin{pmatrix}
		a_{11} & a_{12}\\
		a_{12} & a_{22}
	\end{pmatrix},
	\]
	where
	\[
	\begin{aligned}
		a_{11}
		&=
		1.80+0.25\sin(\pi x)\cos(\pi y)+0.10\cos(2\pi y),\\
		a_{22}
		&=
		1.40+0.20\cos(\pi x)\sin(\pi y)+0.10\sin(2\pi x),\\
		a_{12}
		&=
		0.22\sin(\pi x)\sin(\pi y)
		+0.08\cos(\pi x)\cos(\pi y).
	\end{aligned}
	\]
	The sampled eigenvalue range on \(\Omega_2\) is
	\[
	1.11098\le\lambda(A_2(x,y))\le2.14998.
	\]
	
	For the three-dimensional problem,
	\[
	A_3(x,y,z)
	=
	\begin{pmatrix}
		a_{11} & a_{12} & a_{13}\\
		a_{12} & a_{22} & a_{23}\\
		a_{13} & a_{23} & a_{33}
	\end{pmatrix},
	\]
	with
	\[
	\begin{aligned}
		a_{11}
		&=
		1.80+0.20\sin(\pi x)\cos(\pi y)+0.08\cos(\pi z),\\
		a_{22}
		&=
		1.50+0.18\cos(\pi x)\sin(\pi z)+0.08\sin(\pi y),\\
		a_{33}
		&=
		1.30+0.15\sin(\pi y)\cos(\pi z)+0.06\cos(2\pi x),\\
		a_{12}
		&=
		0.12\sin(\pi x)\sin(\pi y),\\
		a_{13}
		&=
		0.10\cos(\pi y)\sin(\pi z),\\
		a_{23}
		&=
		0.09\sin(\pi x-\pi z)\cos(\pi y).
	\end{aligned}
	\]
	The sampled eigenvalue range on \(\Omega_3\) is
	\[
	1.09026\le\lambda(A_3(x,y,z))\le2.07938.
	\]
	In both cases, the source term
	\[
	q=-\nabla\cdot(A\nabla T)
	\]
	is evaluated by automatic differentiation.
	
	In addition to \(E_{L^2}^{\rm sh}\) and \(E_{H^1}\), we report the relative
	matrix-energy error
	\[
	E_A
	=
	\frac{
		\left(
		\int_\Omega
		A\nabla(v_\theta-T)\cdot\nabla(v_\theta-T)\,\mathrm dx
		\right)^{1/2}
	}{
		\left(
		\int_\Omega
		A\nabla T\cdot\nabla T\,\mathrm dx
		\right)^{1/2}
	}.
	\]
	
	\begin{table}[!htbp]
		\centering
		\caption{Matrix-coefficient heat-conduction results on irregular domains.}
		\label{tab:matrix-heat-results}
		\resizebox{\textwidth}{!}{%
			\begin{tabular}{l c c c c c c}
				\toprule
				case & \(C\) &
				\(E_{L^2}^{\rm sh}\) & \(E_{H^1}\) & \(E_A\) &
				residual & \(B_{L^2}^{\rm sh}\) \\
				\midrule
				2-D matrix heat
				& \(1.40693{\times}10^{-1}\)
				& \(9.49366{\times}10^{-3}\)
				& \(1.86306{\times}10^{-2}\)
				& \(1.77322{\times}10^{-2}\)
				& \(7.332996{\times}10^{-2}\)
				& \(9.08267{\times}10^{-3}\) \\
				3-D matrix heat
				& \(1.53757{\times}10^{-1}\)
				& \(1.04341{\times}10^{-2}\)
				& \(2.54978{\times}10^{-2}\)
				& \(2.47860{\times}10^{-2}\)
				& \(1.12394{\times}10^{-1}\)
				& \(6.23611{\times}10^{-3}\) \\
				\bottomrule
		\end{tabular}}
	\end{table}
	
	\begin{figure}[!htbp]
		\centering
		\includegraphics[height=0.350\textheight]
		{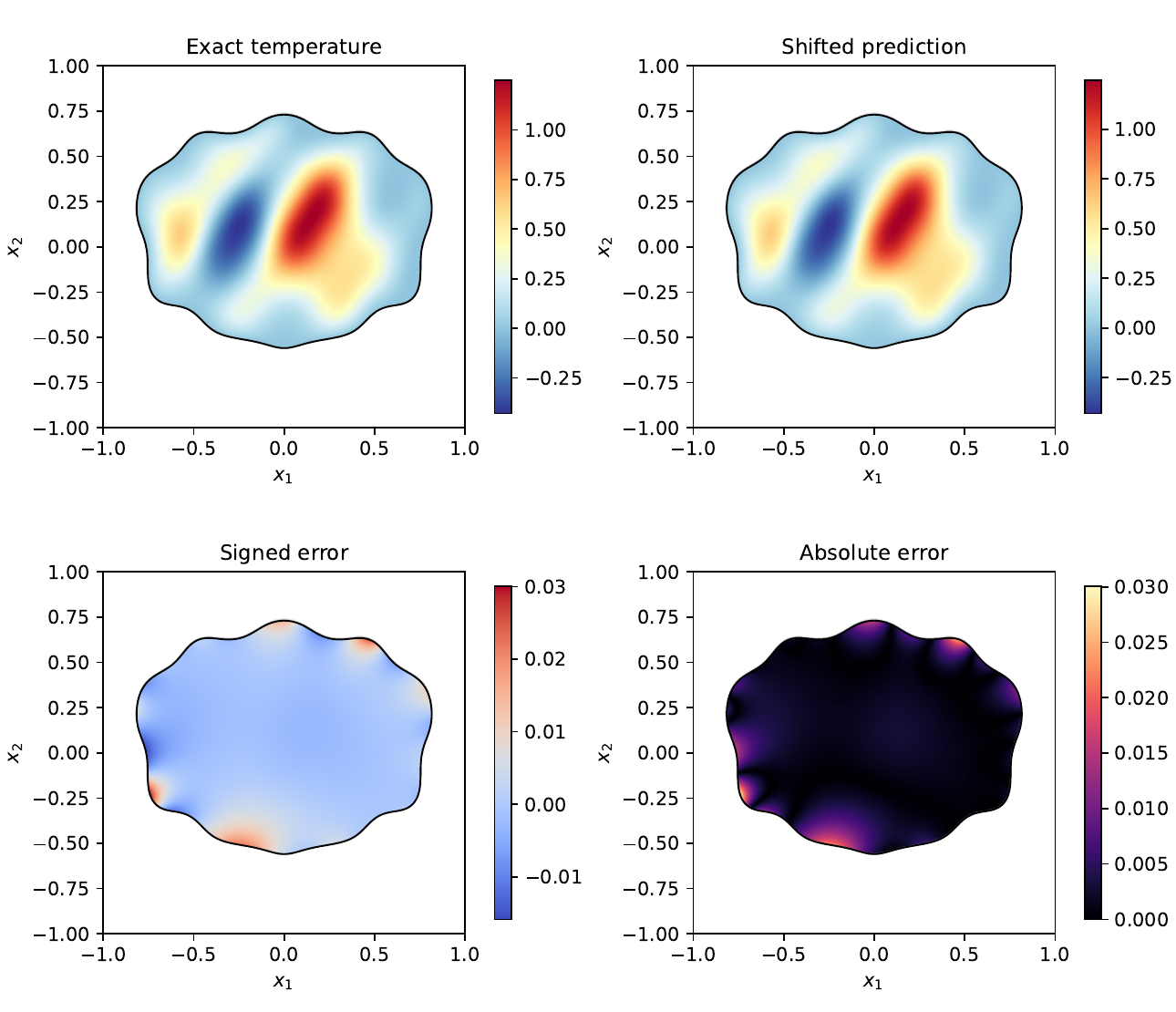}
		\caption{Two-dimensional anisotropic heat-conduction problem on an irregular
			domain.  The panels show the exact temperature, shifted prediction,
			signed error, and absolute error.}
		\label{fig:matrix-heat-2d}
	\end{figure}
	
	\begin{figure}[!htbp]
		\centering
		\includegraphics[height=0.350\textheight]
		{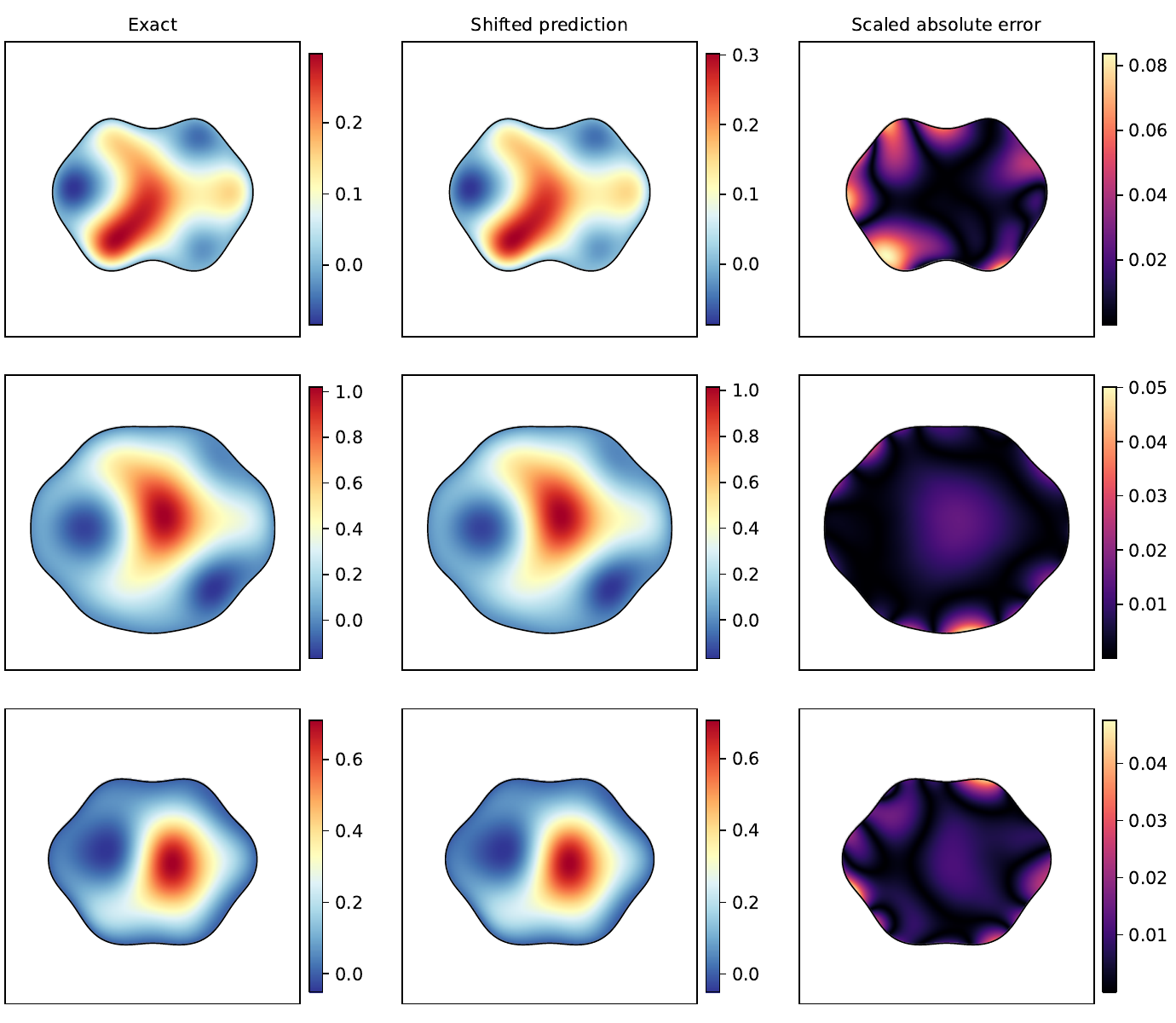}
		\caption{Three-dimensional anisotropic heat-conduction problem.
			Representative slices show the exact temperature, shifted prediction,
			and scaled absolute error.}
		\label{fig:matrix-heat-3d}
	\end{figure}
	
	The matrix-energy errors remain below \(2.5\times10^{-2}\), while the shifted
	relative \(L^2\) errors are approximately \(10^{-2}\).
	
	\paragraph{Interface problems}
	\label{sec:numerics-interface}
	
	In this subsection, we consider an elliptic interface problem to examine the applicability of the proposed formulation in the presence of internal interfaces.  Interface problems involve discontinuous coefficients and transmission conditions across internal interfaces, which introduce additional challenges compared with standard boundary value problems.
	
	The two-network interface formulation is not a direct instance of
	\ref{prop:variable-coefficient-extension}.  A piecewise strong residual does
	not control the interface distribution generated by a flux jump, so the
	continuity and flux conditions are imposed explicitly.
	
	We next combine the outer-boundary formulation with explicit transmission
	conditions.  Let
	\[
	\Omega=(0,1)^4,
	\qquad
	c=\left(\frac12,\ldots,\frac12\right),
	\qquad
	\Gamma=\{x:\|x-c\|=R\},
	\qquad
	R=0.35,
	\]
	and define
	\[
	\Omega^-=\{x:\|x-c\|<R\},
	\qquad
	\Omega^+=\Omega\setminus\overline{\Omega^-}.
	\]
	We consider
	\begin{equation}
		\begin{cases}
			-\nabla\cdot(\beta^\pm\nabla u^\pm)=f^\pm,
			&x\in\Omega^\pm,\\
			u=0,&x\in\partial\Omega,\\
			[u]=0,&x\in\Gamma,\\
			[\beta\partial_nu]=0,&x\in\Gamma,
		\end{cases}
		\label{eq:interface-problem}
	\end{equation}
	with
	\[
	\beta^-=10,
	\qquad
	\beta^+=1.
	\]
	The manufactured solution is defined by
	\[
	p(x)=\prod_{j=1}^4x_j(1-x_j),
	\qquad
	s(x)=\|x-c\|^2-R^2,
	\qquad
	u^\pm(x)=\frac{1000}{\beta^\pm}p(x)s(x).
	\]
	Consequently,
	\[
	f^\pm(x)
	=
	-\beta^\pm\Delta u^\pm(x)
	=
	-1000\,\Delta\bigl(p(x)s(x)\bigr),
	\]
	and the exact solution satisfies the outer homogeneous Dirichlet condition
	and both interface jump conditions.
	
	Two independent subnetworks represent \(u^-_\theta\) and \(u^+_\theta\).
	The outer Dirichlet condition is handled by the BTF-PINN objective, while the
	interface conditions enter through separate loss terms.  The optimized
	functional is
	\begin{equation}
		\begin{aligned}
			\mathcal J(v_\theta)
			={}&
			C_RR_{\rm bulk}(v_\theta)
			+
			\mathcal E_\beta(v_\theta)\\
			&+
			\lambda_DJ_{\Gamma,D}(v_\theta)
			+
			\lambda_NJ_{\Gamma,N}(v_\theta),
		\end{aligned}
		\label{eq:hybrid-interface-loss}
	\end{equation}
	where
	\[
	\begin{aligned}
		R_{\rm bulk}^2
		={}&
		\|-\beta^-\Delta v^-_\theta-f^-\|_{L^2(\Omega^-)}^2
		+
		\|-\beta^+\Delta v^+_\theta-f^+\|_{L^2(\Omega^+)}^2,\\
		\mathcal E_\beta(v_\theta)^2
		={}&
		\beta^-\|\nabla v^-_\theta\|_{L^2(\Omega^-)}^2
		+
		\beta^+\|\nabla v^+_\theta\|_{L^2(\Omega^+)}^2,
	\end{aligned}
	\]
	and
	\[
	J_{\Gamma,D}
	=
	\|v^+_\theta-v^-_\theta\|_{L^2(\Gamma)},
	\qquad
	J_{\Gamma,N}
	=
	\|\beta^+\partial_nv^+_\theta-\beta^-\partial_nv^-_\theta\|_{L^2(\Gamma)}.
	\]
	The residual coefficient \(C_R\) is chosen relative to the weighted energy
	term, while the interface penalties control the continuity and flux
	conditions.  The interface penalty parameters \(\lambda_D\) and \(\lambda_N\) are selected empirically from a prescribed range to balance the interface continuity and
	flux constraints. The reported results correspond to the best-performing
	combination in this search.

	Table~\ref{tab:interface-4d-results} reports the trained model used in
	Figure~\ref{fig:interface-4d-slice}.  This run was selected by the smallest
	shifted \(L^2\) error among the tested configurations.  The bulk residual
	coefficient used in this run is
	\[
	C_R=1.13082\times10^{-1}.
	\]
	
	The relative weighted energy error reported below is
	\[
	E_{\beta,\nabla}
	=
	\frac{
		\left(
		\beta^-\|\nabla(v^-_\theta-u^-)\|_{L^2(\Omega^-)}^2
		+
		\beta^+\|\nabla(v^+_\theta-u^+)\|_{L^2(\Omega^+)}^2
		\right)^{1/2}
	}{
		\left(
		\beta^-\|\nabla u^-\|_{L^2(\Omega^-)}^2
		+
		\beta^+\|\nabla u^+\|_{L^2(\Omega^+)}^2
		\right)^{1/2}
	}.
	\]
	
	\begin{table}[!htbp]
		\centering
		\caption{Four-dimensional hybrid interface result for the trained model
			visualized in Figure~\ref{fig:interface-4d-slice}.  Errors are reported
			after removing the additive constant.}
		\label{tab:interface-4d-results}
		\resizebox{\textwidth}{!}{%
			\begin{tabular}{c c c c c c}
				\toprule
				\(C_R\) & \(E_{L^2}^{\rm sh}\) & \(E_{H^1}\) &
				\(E_{\beta,\nabla}\) & \(B_{L^2}^{\rm sh}\) & flux jump \\
				\midrule
				\(1.13082{\times}10^{-1}\)
				& \(5.44564{\times}10^{-2}\)
				& \(8.86615{\times}10^{-2}\)
				& \(8.77126{\times}10^{-2}\)
				& \(9.47698{\times}10^{-3}\)
				& \(6.71788{\times}10^{-4}\) \\
				\bottomrule
		\end{tabular}}
	\end{table}
	
	\begin{figure}[!htbp]
		\centering
		\includegraphics[width=0.70\textwidth]
		{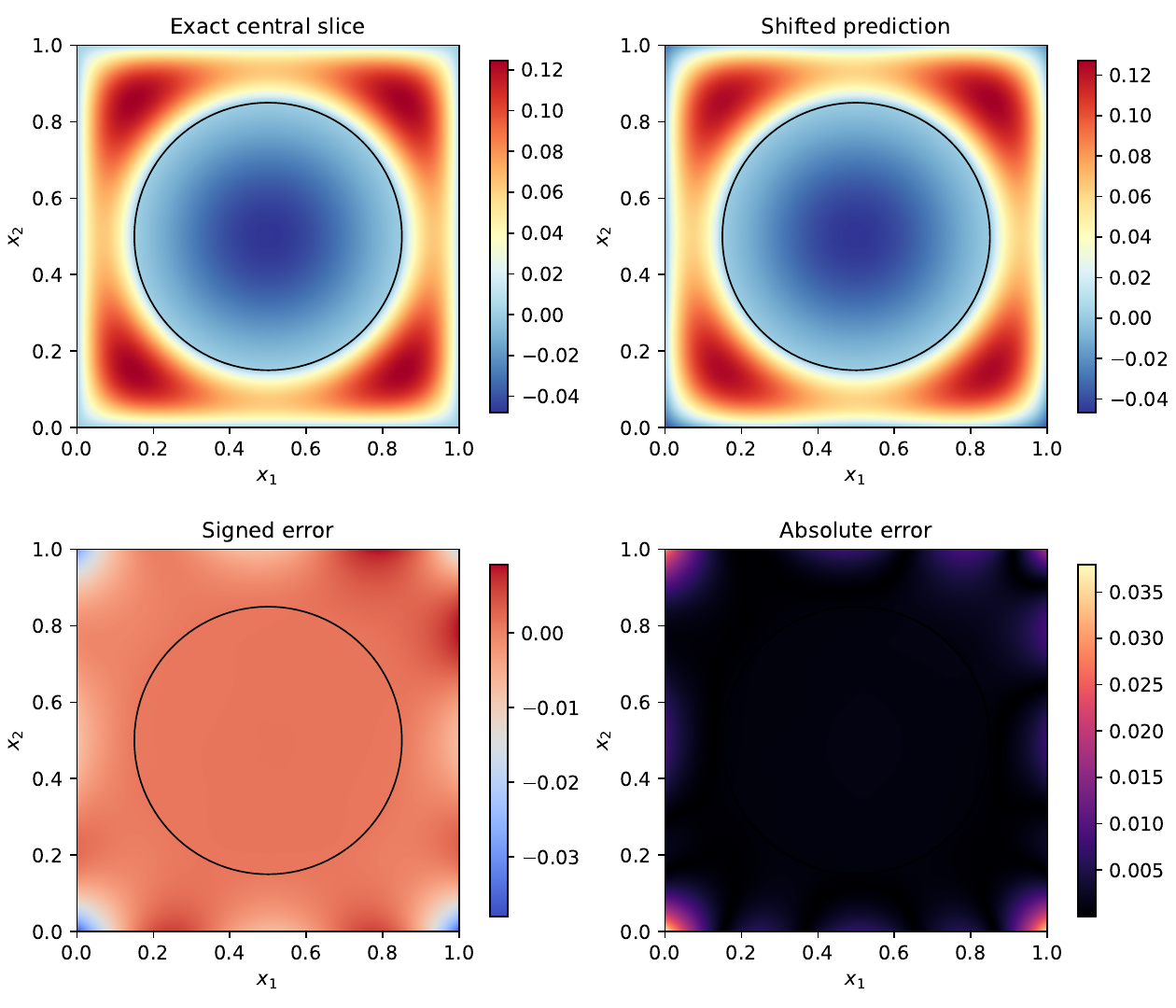}
		\caption{Central two-dimensional slice of the four-dimensional
			interface result reported in Table~\ref{tab:interface-4d-results}.
			The panels show the exact solution, the shifted neural-network
			prediction, the signed error, and the absolute error.  The black curve
			is the intersection of the spherical interface with the slice
			\(x_3=x_4=1/2\).}
		\label{fig:interface-4d-slice}
	\end{figure}
	The result gives a shifted outer-boundary error of
	\(9.48\times10^{-3}\) and a flux jump of \(6.72\times10^{-4}\).

	\paragraph{\(10\)-D Poisson problem}
	\label{sec:10d}
	
	We consider a ten-dimensional product-mode benchmark on
	\[
	\Omega=(0,1)^{10}.
	\]
	The boundary consists of twenty \(9\)-dimensional faces, none of which is used
	in the BTF-PINN training loss.
	
	The manufactured solution is the normalized product mode
	\begin{equation}
		u(x)
		=
		2^{d/2}\prod_{j=1}^{d}\sin(\pi x_j),
		\qquad
		f(x)=d\pi^2u(x),
		\qquad d=10,
		\label{eq:10d-product-solution}
	\end{equation}
	so that \(\|u\|_{L^2(\Omega)}=1\).  The predicted threshold is
	\[
	C_\ast=(\pi\sqrt{10})^{-1}=0.100658,
	\]
	and we use
	\[
	C=0.15=1.490\,C_\ast .
	\]
	
	\begin{table}[!htbp]
		\centering
		\caption{A representative \(10\)-D boundary-training-free run at \(C=0.15\).
			This run is used for the slice visualizations in
			Figures~\ref{fig:10d-slice-diagnostics} and
			\ref{fig:10d-stacked-slices}.}
		\label{tab:10d-representative}
		\begin{tabular}{c c c c c}
			\toprule
			\(C\) & \(C/C_\ast\) & \(E_{H^1}\) & \(E_{L^2}^{\rm sh}\) & \(B_{L^2}^{\rm sh}\) \\
			\midrule
			\(0.150\)
			& \(1.490\)
			& \(1.079\times10^{-1}\)
			& \(6.568\times10^{-2}\)
			& \(8.245\times10^{-2}\) \\
			\bottomrule
		\end{tabular}
	\end{table}
	
	Table~\ref{tab:10d-representative} reports one representative run at
	\(C=0.15\).  The relative \(H^1\)-seminorm error is \(10.79\%\), the shifted
	relative \(L^2\) error is \(6.57\%\), and the shifted posterior boundary
	\(L^2\) error is \(8.25\%\).  This run serves as the visualization checkpoint;
	robustness is assessed separately over multiple independent runs in
	Table~\ref{tab:10d-robustness}.
	
	To visualize the trained high-dimensional field, we inspect the trained
	\(10\)-D network from Table~\ref{tab:10d-representative} through
	lower-dimensional sections.  The first visualization fixes
	\(x_3=\cdots=x_{10}=0.5\) and displays the \((x_1,x_2)\)-slice of the aligned
	prediction, the exact solution, the absolute pointwise error, and the
	pointwise relative error.
	
	\begin{figure}[!htbp]
		\centering
		\IfFileExists{fig_heatmap_grid.pdf}{%
			\includegraphics[height=0.4\textheight,keepaspectratio]{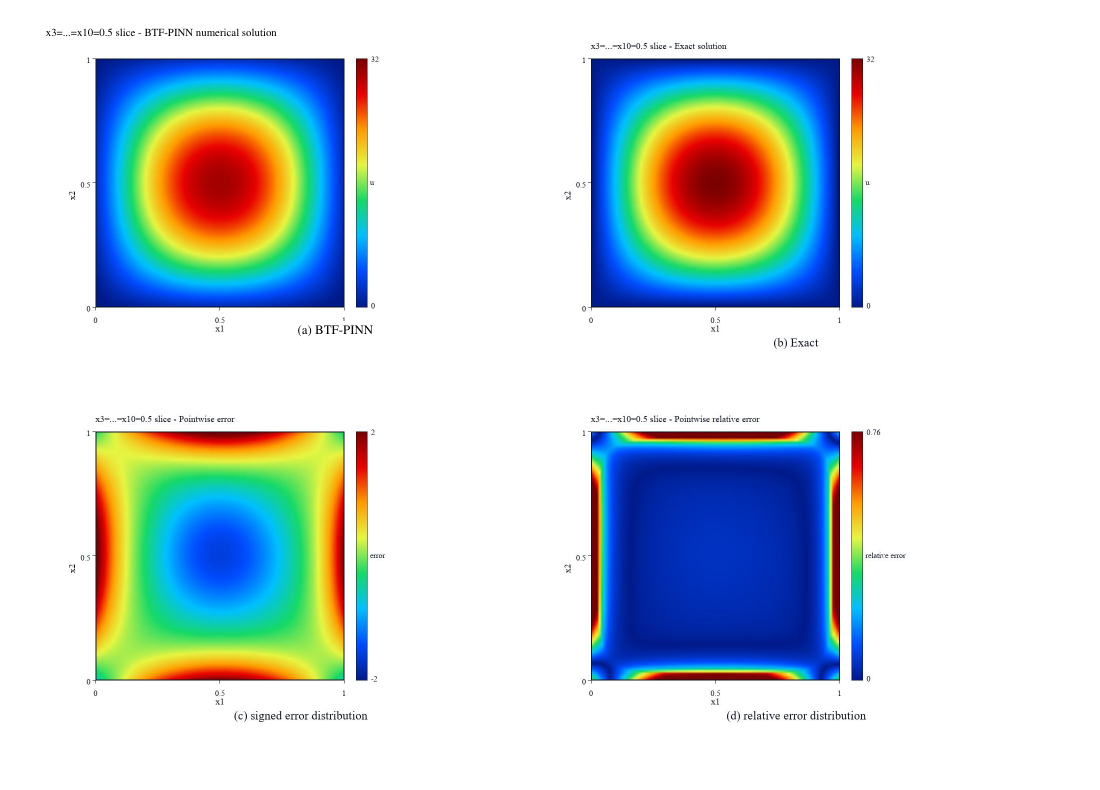}%
		}{%
			\fbox{\parbox{0.70\textwidth}{\centering Missing figure file:
					\texttt{fig\_2607\_style\_10d\_heatmap\_grid.pdf}.}}%
		}
		\caption{Two-dimensional slice diagnostics for the representative \(10\)-D
			run in Table~\ref{tab:10d-representative}.  The panels display the aligned
			BTF-PINN prediction, the exact solution, the absolute pointwise error, and
			the pointwise relative error.}
		\label{fig:10d-slice-diagnostics}
	\end{figure}
	
	A complementary three-dimensional view is obtained by fixing
	\(x_4=\cdots=x_{10}=0.5\) and plotting stacked slices in the
	\((x_1,x_2,x_3)\) variables.  This is not a separate three-dimensional problem, but a visualization of the
	same \(10\)-D trained network on a three-dimensional section of the cube.
	
	\begin{figure}[!htbp]
		\centering
		\IfFileExists{fig_3d_stacked_slices.pdf}{%
			\includegraphics[height=0.4\textheight,keepaspectratio]{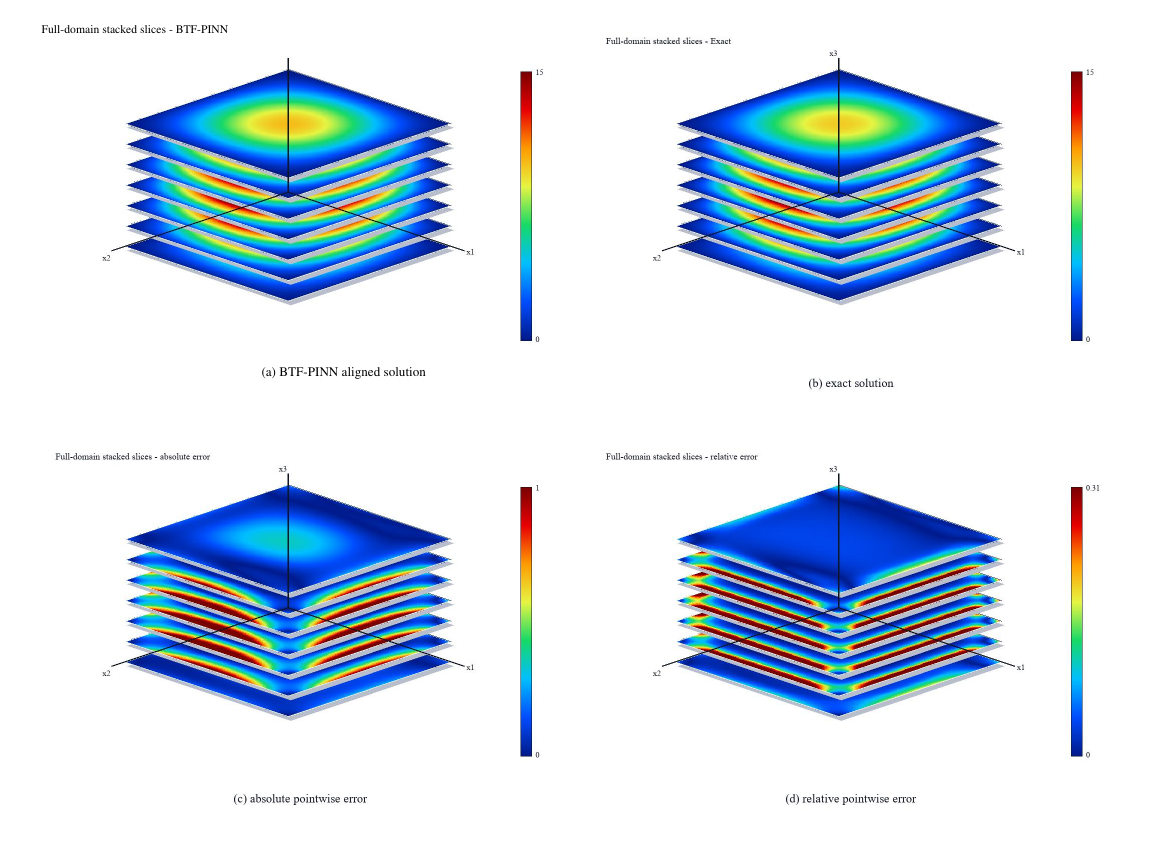}%
		}{%
			\fbox{\parbox{0.70\textwidth}{\centering Missing figure file:
					\texttt{fig\_2607\_style\_3d\_stacked\_slices.pdf}.}}%
		}
		\caption{Stacked-slice visualization of the representative \(10\)-D run in
			Table~\ref{tab:10d-representative}.  Several \((x_1,x_2)\)-slices are shown
			along the \(x_3\)-direction after fixing \(x_4,\ldots,x_{10}=0.5\).}
		\label{fig:10d-stacked-slices}
	\end{figure}
	
	\medskip
	\noindent\textbf{Robustness over independent runs.}
	We also report four independent runs at \(C=0.15\), all obtained with the
	same training protocol.
	
	\begin{table}[!htbp]
		\centering
		\caption{\(10\)-D robustness at \(C=0.15\).  Entries are mean \(\pm\) standard deviation over four independently obtained runs.  The boundary is used only for posterior diagnostics.}
		\label{tab:10d-robustness}
		\begin{tabular}{c c c c }
			\toprule
			runs & \(C/C_\ast\) & \(E_{H^1}\) & \(E_{L^2}^{\rm sh}\) \\
			\midrule
			\(4\)
			& \(1.490\)
			& \(1.602\times10^{-1}\pm3.546\times10^{-2}\)
			& \(1.078\times10^{-1}\pm2.940\times10^{-2}\)
			\\
			\bottomrule
		\end{tabular}
	\end{table}
	
	Across four independent runs, the mean relative \(H^1\)-seminorm error is
	\(1.602\times10^{-1}\) and the mean shifted relative \(L^2\) error is
	\(1.078\times10^{-1}\).  The shifted posterior boundary error is of the same
	order.  Thus the high-dimensional behavior is not an isolated favorable
	snapshot. The observed variance across independent runs can be attributed to the
	non-convex loss landscape and the inherent uncertainty of neural-network optimization.
	
	\paragraph{Summary of numerical results}
	The numerical results demonstrate that BTF-PINN can accurately solve
	homogeneous Dirichlet problems without boundary sampling or boundary
	penalties, including high-dimensional examples where boundary discretization
	becomes increasingly costly.  The comparison with boundary-penalty PINNs shows
	that the proposed interior-only formulation can recover boundary behavior
	comparable to that of explicitly constrained methods.  The residual-weight study
	confirms the predicted threshold behavior, with stable accuracy obtained when
	the weight is chosen above the theoretical scale.  Additional experiments on
	matrix-coefficient elliptic problems, irregular domains, and interface
	problems further illustrate the applicability of the proposed interior
	variational strategy.

	\FloatBarrier

    \section{Conclusion}
	\label{sec:conclusion}
	
	We introduced a boundary-free variational principle for homogeneous Dirichlet
	problems and developed its neural implementation, BTF-PINN. The proposed
	formulation characterizes the Dirichlet boundary condition through a
	purely interior functional.
	The harmonic decomposition of the error provides the underlying mechanism: the
	residual term controls the zero-trace component, while the Dirichlet seminorm
	controls the harmonic component up to an additive constant. For
	\(C>C_P\), the Dirichlet solution is uniquely characterized up to an additive constant.
	
	The proposed variational principle is established at the level of general
	trial classes rather than being restricted to neural networks. Based on the
	error control established by the variational principle, we develop a convergence analysis framework involving approximation properties of the trial class and \(\varepsilon\)-optimal
	solutions, thereby avoiding the assumption that the infimum is attained in the neural
	optimization problem.
	
	More broadly, the underlying boundary-free formulation of Dirichlet problems is not confined to the $H^1$ setting, but opens a natural pathway toward higher-order $H^m$ problems, suggesting a unified perspective on the treatment of essential boundary conditions across different orders.

	\section*{Acknowledgments}
	
	This work was partially supported by the National Natural Science Foundation of China (NSFC) under Grant Nos. 92370205 and 12271512. 
	
	\bibliographystyle{cas-model2-names}
	
	\bibliography{btfpinn}
\end{document}